\documentclass[11pt]{article}
\usepackage[utf8]{inputenc}
\usepackage{color}
\usepackage{graphicx}

\usepackage{amsthm}
\usepackage{amssymb}
\usepackage{amsmath}
\usepackage{float}
\usepackage{bm}
\usepackage{subcaption}
\newtheorem{theorem}{Theorem}
\newtheorem{remark}{Remark}
\newtheorem{corollary}{Corollary}
\newtheorem{definition}{Definition}
\newtheorem{lemma}{Lemma}

\usepackage{graphicx}

\graphicspath{{figs/}{DNA_simulation/3layers/}}

\begin{document}
\title{Asymptotic Error Analysis of Gauss Quadrature for \\ Nonsmooth Functions\thanks{P.L. was partially supported by NSF grant DMS-2318053.}}

\author{Pei Liu\thanks{Department of Mathematics and Systems Engineering, Florida Institute of Technology, Melbourne, FL, 32901.(pliu@fit.edu)}
}

\date{}
\maketitle

\begin{abstract}
We derive an asymptotic error formula for Gauss--Legendre quadrature applied to functions with limited regularity, using the contour-integral representation of the remainder term. To address the absence of uniformly valid approximations of Legendre functions near $[-1,1]$, we approximate the integrand by smoother functions with singularities displaced from the interval and then obtain asymptotic expansions of the Legendre functions that hold uniformly along the contour. The resulting error formula identifies not only the optimal convergence rate but also the leading coefficient, expressed in terms of $\cos((2n+1)\phi)$, where $n$ is the number of quadrature points and $\cos(\phi)$ locates the singularity. This characterization enables both the selection of quadrature sizes that minimize the leading error and the use of the error formula as a correction term to accelerate convergence. Applications to functions with power and logarithmic singularities are presented, and numerical experiments confirm the accuracy of the analysis.
\end{abstract}

\section{Introduction}\label{sec:introduction}

Gaussian quadrature rule plays a central role in numerical integration, particularly due to its optimality in integrating polynomials and its rapid convergence. The $n$-point Gauss--Legendre quadrature is given by,
\begin{equation}
    \int_{-1}^1 f(x) dx \approx \sum_{j=1}^n \omega_j f(x_j),
\end{equation}
which is exact for all polynomials of degree up to $2n-1$. Here $x_j$ denotes the $j-$th root of Legendre polynomial $P_n(x)$, and the associated weight $\omega_j = \int_{-1}^1 \frac{P_n(x)}{(x - x_j) P_n'(x_j)} dx =  \frac{2}{(1-x_j^2)[P'_n(x_j)]^2}$. For a function $f(z)$ analytic in a Bernstein ellipse in the complex plane with foci at $\pm 1$, the quadrature remainder admits the contour integral representation,
\begin{equation}
R_n[f] \triangleq \int_{-1}^1 f(x) dx - \sum_{j=1}^n w_j f(x_j) = \frac{1}{\pi i} \int_C f(s) \frac{Q_n(s)}{P_n(s)} ds. \label{Gauss_Error_Contour}
\end{equation}
The function $Q_n(s) = \frac{1}{2} \int_{-1}^1 \frac{P_n(x)}{s - x} dx$ is the Legendre function of the second kind. The contour $C$ is inside the Bernstein ellipse and encloses the real interval $[-1,1]$. Analogous representations hold for other Gaussian quadratures with weights (Jacobi, Chebyshev, Laguerre, Hermite).

For analytical integrand, or integrand with singularity (such as pole and branch points) away from the real interval $[-1,1]$, classical asymptotic analysis of contour integration have shown that the quadrature error decays exponentially \cite{barrett1961convergence, chawla1968asymptotic, chawla1969davis, gautschi1983error}. This approach relies on the uniform asymptotic approximation of Legendre' functions for $ z \in \mathbb{C} \setminus U_\varepsilon([-1,1]) $, where $ U_\varepsilon([-1,1])$ denotes a fixed neighborhood of $[-1,1]$,
\begin{equation}
\begin{cases}
    P_n(z) = \frac{1}{(z^2-1)^{1/4}}  \frac{(z + \sqrt{(z^2-1})^{n+1/2}}{\sqrt{\pi( 2n+1)}} \left( 1 + O(\frac{1}{n}) \right),\\ 
    Q_n(z) = \frac{1}{(z^2-1)^{1/4}}  \frac{(z - \sqrt{(z^2-1})^{n+1/2}}{\sqrt{( 2n+1)/\pi}} \left( 1 + O(\frac{1}{n}) \right).\label{PQ_large}
\end{cases}
\end{equation}
Here the principal branch of $\sqrt{z^2-1}$ is chosen such that it is positive when $z>1$, and continuous on the $z$-plane along the branch cut $[-1,1]$. These results relate the convergence rate to the distance of the nearest singularity from the real axis, offering sharp and interpretable error predictions for those functions with good regularity.

In many practical applications, such as in boundary integral/element methods, fracture mechanics and electromagnetic scattering \cite{elliott2008clenshaw,huybrechs2009generalized,liang2021high,gupta2022finite,nesvit2014scattering,kayijuka2021fast}, the integrand is not analytic but merely continuous or of bounded variation. Understanding this behavior is essential both for guiding practical applications and for designing modified quadratures that recover higher accuracy. In literature, various of approaches \cite{sidi2011asymptotic,zhao2013sharp,huybrechs2009generalized,diethelm2014error,weideman2019gauss,kazashi2023suboptimality,goda2024sharp,trefethen2022exactness} have been developed to study the Gaussian quadrature error for nonsmooth integrands, such as using Chebyshev expansion \cite{curtis1972gaussian, riess1971error, trefethen2008gauss, trefethen2019approximation, xiang2012convergence, xiang2012error, xiang2016improved} and Peano kernel \cite{brass1997gaussian,petras1993asymptotics,brass2011quadrature,forster1991error,diethelm1996peano}. In such settings, Gaussian quadrature remains indispensable, but the error decays only algebraically, with convergence rates determined by the smoothness class of the integrand.


In this work, we aim to extend the usage of Eq. \eqref{Gauss_Error_Contour} to derive sharp asymptotic formulas for the error of Gauss quadrature when applied to functions with limited regularity. For such functions, their analytic continuations develop singularities on the real interval $(-1,1)$. These singularities pose a major challenge because they coincide with the branch cut of the Legendre functions, causing the contour $C$ to intersect with $(-1,1)$. This challenge is further complicated by the absence of uniformly valid asymptotic expansions for Legendre functions near $(-1,1)$, which obstructs precise kernel analysis and prevents straightforward estimation of the quadrature error. In \cite{kutz1984asymptotic}, Kutz discussed splitting the contour into two parts to estimate the convergence order for the integrand $|x-a|^\alpha$.

To address these challenges, we  consider the following strategy. First, we introduce a series of smoother functions $f_n(x)$ to approximate the integrand $f(x)$ of limited regularity, aiming to move the singularity away from the real interval $[-1,1]$. With proper choice of $f_n(x)$, and let $h_n(x)=f_n(x)-f(x)$, $R_n[h_n]$ could be controlled as higher order terms. Second, the error of $n$-point Gauss quadrature applied to $f_n(x)$ could be parametrized as a contour integral that depends on $n$. A uniform asymptotic approximation for $P_n(z)$ and $Q_n(z)$ on the contour is developed to further simplify the contour integral to a definite integral, allowing us to extract sharp asymptotics for $R_n[f_n(x)]$, and hence for $R_n[f(x)]$.

For a nonsmooth real function with finitely many isolated singularities on $(-3,3)$, it can be decomposed into a sum of components, each associated with a single singularity. Since singularities outside a Bernstein ellipse for $[-1,1]$ only affect the exponential convergence of Gaussian quadrature, it suffices to consider a single singularity in the interior of $[-1,1]$. In this paper, we focus on such interior singularities, omit the case of endpoint singularities and assume the following properties.
\begin{definition} \label{definition}
Let $k\geq 0$ be an integer, $\alpha \in (0,1]$. We say that a function $f(x)$ belongs to the class $D_b^{k,\alpha}[-1,1]$ if the following hold:  
\begin{enumerate}
    \item   
    $f(x)$ is $k$-times continuously differentiable on the interval $[b-c,b+c]$, with $b\in (-1,1)$ and $c=\max(1-b,1+b)$.
    \item There exists an analytic continuation $f(z)$ in a neighborhood of $[b-c,b+c]$, except along the branch cut $(b- i\infty, b+i\infty)$.
    \item $\displaystyle \limsup_{z \to b} |z-b|^{1-\alpha} |f^{(k+1)}(z)| < \infty$.
\end{enumerate}
\end{definition}
The branch cut $(b - i\infty,\, b + i\infty)$ will be discussed in Lemma~\ref{f_decomposition}. The last condition controls the singular behavior of the integrand on real axis, so that $D_b^{k,\alpha}[-1,1]$ is a subset of the H\"older space $C^{k,\alpha}$, thus a subset of Chebyshev space $X^{k+\alpha}$ as defined in \cite{trefethen2019approximation, xiang2016improved}. Moreover, it also guarantees that the complex function $f(z)$ remains integrable near the branch point,  thereby excluding essential singularities.

The main theorem is summarized below.
\begin{theorem} \label{main_theorem}
    For $f(x) \in D_b^{k,\alpha}[-1,1]$, let $\phi = \arccos(b)$ and $\Psi = (2n+1)\phi - \frac{\pi}{2}$. The asymptotic remainder of the $n$-point Gauss quadrature applied to $f(x)$ is,
    \begin{eqnarray}
        R_n[f(x)] &=& \frac{1}{ n} \int_{0}^{M \log n}  \Re\left(\left[ f(b+i\frac{y}{n})\right] \frac{  i\exp\left(\frac{-2y}{\sin\phi}\right) + i\cos \Psi + \sin \Psi  }{\cosh\left(\frac{2y}{\sin\phi}\right) + \cos \Psi}\right)
        dy \label{leading_order_theorem} \\ &&+\begin{cases}
        O(\frac{1}{n^{2\alpha+2k+1}}), \ \ \ \ \  \text{ if } k+\alpha <1,\\
        O(\frac{\log n}{n^{k+\alpha+2}}), \ \ \ \ \ \ \ \text{ if } k+\alpha =1,\\
        O(\frac{1}{n^{k+\alpha+2}}),\ \ \ \ \ \ \  \text{ if }  k+\alpha >1.
        \end{cases} \label{higher_order_estimate}
    \end{eqnarray}
    Here $\left[ f(b+i\frac{y}{n})\right]=f(b^++i\frac{y}{n})-f(b^-+i\frac{y}{n})$ represents the jump across the branch cut. The analytical continuation of $f(z)$ can be given by Eq. \eqref{f_analytical_continuation}. $\Re$ represents the real part. $M$ is a fixed number that chosen to be sufficiently large.
\end{theorem}

\begin{remark}\label{main_order}
    Definition \ref{definition} implies $\left[ f(b+i\frac{y}{n})\right]\leq O(\frac{y^{k+\alpha}}{n^{k+\alpha}})$, thus the leading order term given in Eq. \eqref{leading_order_theorem} is $\leq O(\frac{1}{n^{k+\alpha+1}})$.
\end{remark}

\begin{remark}
The integral in Eq. \eqref{leading_order_theorem} may be evaluated over $[0,\infty)$, understood in the asymptotic sense as  $n \to \infty$ with $\frac{M \log n}{n} \to 0$. \label{integral_infty}
\end{remark}

\begin{remark}
The leading-order term in Eq.~\eqref{leading_order_theorem} is determined solely by the local behavior of $f(x)$ near the singularity $x=b$. So, the 2nd condition in Definition \ref{definition} may be relaxed to: there exists an analytic continuation $f(z)$ in a neighborhood of $[b-\delta,b+\delta]$, except along the branch cut, with $\delta$ being an arbitrary positive number. In such a case, Eq.~\eqref{leading_order_theorem} can be applied to estimate the error contribution from each singular point.
\end{remark}
\begin{remark}
    Eq. \eqref{leading_order_theorem} can be used as a correction term in Gauss quadrature, so that the error convergence is improved to Eq. \eqref{higher_order_estimate}. See Example \ref{example5}.
\end{remark}

Throughout the paper, we adopt the standard $O$ and $o$ notations, following the classical references \cite{szeg1939orthogonal,olver1997asymptotics,olver2010nist}:
\begin{equation}
    f_n \sim g_n \Longleftrightarrow \lim_{n \to \infty}\frac{f_n}{g_n} = 1, \ \ f_n \leq  O\left(g_n\right)   \Longleftrightarrow \limsup_{n\to \infty} \left|\frac{f_n}{g_n}\right|  < \infty.
\end{equation}
\begin{eqnarray}
    &&f_n(x) \leq  O\left(g_n(x))\right) \text{ uniformly for } x \in S_n  \nonumber\\ &&\Longleftrightarrow \limsup_{n\to \infty} \sup_{x \in S_n}\left|\frac{f_n(x)}{g_n(x)}\right|  < \infty.
\end{eqnarray}
\begin{eqnarray}
    &&f_n(x) = g_n(x) + O(h_n(x)) \text{ uniformly for } x \in S_n  \nonumber\\ &&\Longleftrightarrow \limsup_{n\to \infty} \sup_{x \in S_n}\left|\frac{f_n(x) -g_n(x)}{h_n(x)}\right|  < \infty.
\end{eqnarray}

The rest of the paper is organized as follows. Section~\ref{sec:asymptotic} introduces the uniform asymptotic approximations of Legendre functions, which forms the basis for the proof of the main theorem. Section~\ref{sec:proof} proves the main theorem and applies the theorem to functions with power and logarithmic singularities. Section~\ref{sec:numerics} presents numerical examples and discussions.

\section{Asymptotic Approximations of Legendre Functions}\label{sec:asymptotic}
Extensive work exists in the literature on uniform asymptotic expansions of Legendre functions, for example \cite{dunster2003uniform,szeg1939orthogonal,olver1997asymptotics,olver2010nist,bakaleinikov2020uniform,nemes2020large} and the references therein. However, these formulas typically require either that $z$ being real or that it maintain a positive separation from the interval $[-1,1]$. Consequently, they cannot be directly applied to our setting. 

We introduce the following lemmas and corollaries; their proofs are provided in the appendix.
\begin{lemma} \label{lemma_P_uniform}
    The following asymptotic approximation is uniform in $\displaystyle z \in \Omega_n = \left\{  z = \cosh(\xi) \left| \bigl|\cosh\!\bigl((n+\tfrac{1}{2})\xi\bigr)\bigr| \;\geq\; \frac{L}{n}, |\sinh(\xi)| \geq \epsilon\right\} \right.$, as $n \to \infty$, with $L$ and $\epsilon$ being arbitrary positive constants,
    \begin{eqnarray}
    P_n(\cosh \xi) =&& \sqrt{\frac{i}{2n\pi \sinh \xi}}\left[\left(1 - \frac{1}{4n} + \frac{\coth(\xi)}{8n}\right)\exp\left((n+\frac{1}{2})\xi-i\frac{\pi}{4}\right) \right. \nonumber\\
    && \left.  + \left(1 - \frac{1}{4n} - \frac{\coth(\xi)}{8n}\right)\exp\left(-(n+\frac{1}{2})\xi+i\frac{\pi}{4}\right)\right] \left( 1 + O(\frac{1}{n}) \right). \label{P_asymptotic_uniform}
\end{eqnarray}
Here the principal branch of $\xi = \log(z + \sqrt{z^2-1})$ is chosen such that $\xi$ is positive when $z>1$, and continuous on the $z$-plane along the branch cut $[-1,1]$.
\end{lemma}

\begin{remark}
    When $ z \in \mathbb{C} \setminus U_\varepsilon([-1,1]) $, Eq. \eqref{P_asymptotic_uniform} is consistent with Eq. \eqref{PQ_large}. When $z= \cos \phi \in (-1,1)$, Eq. \eqref{P_asymptotic_uniform} is consistent with Hilb's approximation,
    \begin{equation}
    \displaystyle P_n(\cos \phi) = \sqrt{\frac{2}{n \pi \sin \phi }}\cos\left( (n+\frac{1}{2})\phi -  \frac{\pi}{4} \right)  + O(\frac{1}{n^{\frac{3}{2}}}). \label{P_real}
\end{equation}
\end{remark}
\begin{lemma}
The following asymptotic approximation is uniform in $\displaystyle z \in \Omega'_n = \left\{  z = \cosh(\xi) \left| |\sinh(\xi)| \geq \epsilon\right\} \right.$, as $n \to \infty$,
    \begin{equation}
    Q_n(\cosh(\xi)) = \sqrt{\frac{\pi}{2(n+1)i \sinh \xi}}\exp\left(-(n+\frac{1}{2})\xi+i\frac{\pi}{4}\right)  \left( 1 + O(\frac{1}{n}) \right) \label{Q_asymptotic_uniform}
\end{equation}
\end{lemma}

\begin{corollary} \label{Q_P_ratio_bernstein}
    For fixed $M>0$, on the Bernstein ellipse
\begin{equation}
    \Omega_B = \left\{ z \in \mathbb{C} \Big| \big| z + \sqrt{z^2 - 1} \big| = 1 + \frac{M \log n}{n}\right\},
\end{equation}
the asymptotic ratio between $Q_n(z)$ and $P_n(z)$ is uniformly bounded by,
    \begin{equation}
        \left|\frac{Q_n(z)}{P_n(z)} \right| \leq O(\frac{1}{n^{2M}}).
    \end{equation} 
\end{corollary}

\begin{corollary} \label{Q_P_ratio_corollary}
With given $b = \cos \phi \in (-1,1)$, and arbitrary $M >0$, the following asymptotic approximation is uniform for all $\displaystyle y \in \left[\frac{1}{n}, M \log n\right]$, as $n \to \infty$,
\begin{equation}
    \frac{Q_n(b +\frac{i}{n}y)}{P_n(b +\frac{i}{n}y)} = \frac{i \pi}{\exp\left(\frac{2y}{\sin\phi}+i\Psi\right) + 1} + O\left(\frac{1}{n \left(\exp\left(\frac{2y}{\sin\phi}\right)-1\right)^2}\right). \label{Q_P_ratio}
\end{equation}
\end{corollary}

\section{Gauss Quadrature Error Asymptotes} \label{sec:proof}\begin{lemma} \label{f_decomposition}
    Let $f(x) \in D_b^{k,\alpha}[-1,1]$, with one singular point $x=b \in (-1,1)$, let $c = \max(1-b,1+b)$, then $f(x)$ can be represented in the form,
    \begin{equation}
        f(x) = P_{k}(x) + (x-b)^k g_1(|x-b|) + (x-b)^{k+1} g_2(|x-b|),
    \end{equation}
    \begin{itemize}
        \item $P_k(x)$ is a polynomial of degree $k$.
        \item $g_1(x)$ is analytic on $(0,c]$, \begin{equation}
            \limsup_{x\to 0^+} x^{l-\alpha} |g_1^{(l)}(x)| < \infty, \text{ for } l=0,1,\cdots,k+1. \label{g_1_condition}
        \end{equation}
        \item $g_2(x)$ is analytic on $(0,c]$, \begin{equation}
            \limsup_{x\to 0^+} x^{l+1-\alpha} |g_2^{(l)}(x)| < \infty, \text{ for } l=0,1,\cdots,k+1. \label{g_2_condition}
        \end{equation}
    \end{itemize}
    Moreover, the analytical continuation near $[b-c,b+c]$ is given by,
    \begin{equation}
        f(z) = P_{k}(z) + (z-b)^k g_1(\sqrt{(z-b)^2}) + (z-b)^{k+1} g_2(\sqrt{(z-b)^2}), \label{f_analytical_continuation}
    \end{equation}
    with branch cut $(b-i\infty,b+i\infty)$, so that $\sqrt{(x-b)^2}=|x-b|$ on the real axis.
\end{lemma}
\begin{proof} 
Let \begin{equation}
    P_k(x) = \sum_{s=0}^k f^{(s)}(b)\frac{(x-b)^s}{s!}, \ \ g(x) = \frac{f(x)-P_k(x)}{(x-b)^k}.
\end{equation} 
Then $g(x)$ is analytic on $[b-c,b)$ and $(b,b+c]$. For $l =0,\cdots,k+1$,
\begin{equation}
    \limsup_{x\to b} \frac{|\left((x-b)^k g(x)\right)^{(l)}|}{|x-b|^{\alpha+k-l}}=\limsup_{x\to b} \frac{|f^{(l)}(x)-P^{(l)}(x)|}{|x-b|^{\alpha+k-l}} < \infty,
\end{equation}
Thus one can apply induction from $l=0$ to $l =k+1$,
\begin{equation}
    \limsup_{x\to b} (x-b)^{l-\alpha} |g^{(l)}(x)| < \infty, \text{ for } l =0,\cdots,k+1.
\end{equation}
Define even functions
\begin{equation}
    g_1(x) = \frac{g(b+x) + g(b-x)}{2}, \ g_2(x) = \frac{g(b+x) - g(b-x)}{2x}.
\end{equation}
Then,
\begin{eqnarray}
    &&\limsup_{x\to 0^+} x^{l-\alpha} |g^{(l)}_1(x)| \nonumber\\
    &\leq& \limsup_{x\to b^+}   |x-b|^{l-\alpha}\left|\frac{g^{(l)}(x)}{2}\right| + \limsup_{x\to b^-}   |x-b|^{l-\alpha}\left|\frac{g^{(l)}(x)}{2}\right| < \infty,
\end{eqnarray}
and
\begin{eqnarray}
    &&\limsup_{x\to 0^+} x^{l-\alpha} |(xg_2(x))^{(l)}| \nonumber\\
    &\leq& \limsup_{x\to b^+}   |x-b|^{l-\alpha}\left|\frac{g^{(l)}(x)}{2}\right| + \limsup_{x\to b^-}   |x-b|^{l-\alpha}\left|\frac{g^{(l)}(x)}{2}\right| < \infty. \label{xg2_regularity}
\end{eqnarray}
Apply induction on Eq. \eqref{xg2_regularity} from $l=0$ to $l =k+1$, further imply,
\begin{equation}
    \limsup_{x\to 0^+} x^{l+1-\alpha} |g_2^{(l)}(x)| < \infty, \text{ for } l = 0, \cdots, k+1.
\end{equation}
Similarly, let $f(z)$ be the analytical continuation in a neighborhood of $[b-c,b+c]$, and define 
\begin{equation}
    g(z) = \frac{f(z)-P_k(z)}{(z-b)^k}, \ \ \tilde{g}_1(z) = \frac{g(b+z) + g(b-z)}{2},\ \ \tilde{g}_2(z) = \frac{g(b+z) - g(b-z)}{2z}. \nonumber
\end{equation}
Then $\tilde{g}_1(x) = g_1(|x|)$, $\tilde{g}_2(x) = g_2(|x|)$ on the real axis $[-c,c]$. According to identity theorem, with chosen branch cut, $\tilde{g}_1(z) = g_1(\sqrt{z^2})$, $\tilde{g}_2(z) = g_2(\sqrt{z^2})$.
\end{proof}

\begin{lemma} \label{lemma_g1}
    Consider the $D_b^{k,\alpha}[-1,1]$ function of the form, 
    \begin{equation}
        f(x) = (x-b)^k g(|x-b|),
    \end{equation}
    with $g(x)$ satisfying Eq. \eqref{g_1_condition}. Let
    \begin{equation}
        f_n(x) = (x-b)^k g(\sqrt{(x-b)^2+\frac{1}{n^4}}), \ \ \ h_n(x) = f_n(x) - f(x).
    \end{equation}
    Then the remainder of $n$-point Gauss quadrature on real interval $[-1,1]$  satisfies,
    \begin{itemize}
        \item[(I)] 
        $\displaystyle |R_n[h_n]| \leq  \begin{cases}
        \displaystyle O(\frac{1}{n^{k+\alpha+2}}), \ \ \ \ \text{ if } k \geq 1,\\
        \displaystyle O(\frac{1}{n^{2\alpha+1}}), \ \ \ \ \ \ \text{ if } k=0.
    \end{cases}$
    \item[(II)] Let $b=\cos\phi$, $\Psi =(2n+1)\phi-\frac{\pi}{2}$. Then 
    \begin{eqnarray}
        R_n[f] &=& \frac{2}{ n^{k+\alpha+1}}\int_{0}^{M\log n}  \Re\left(\frac{ i^{k+1}y^{k} n^{\alpha} \left[ g(i\frac{y}{n})\right] }{\exp\left(\frac{2y}{\sin\phi}+i\Psi\right) + 1}\right) dy \nonumber\\
        &&+\begin{cases}
        O(\frac{1}{n^{2\alpha+2k+1}}), \ \ \ \ \  \text{ if } k+\alpha <1,\\
        O(\frac{\log n}{n^{k+\alpha+2}}), \ \ \ \ \ \ \ \text{ if } k+\alpha =1,\\
        O(\frac{1}{n^{k+\alpha+2}}),\ \ \ \ \ \ \  \text{ if }  k+\alpha >1.
    \end{cases}
    \end{eqnarray}
    Here $\left[g(iy)\right] = g(iy)-g(-iy)$ represents the jump across the branch cut. $M$ is a fixed number that chosen to be sufficiently large.
    \end{itemize}
\end{lemma} 
\begin{proof} {\textit Part (I).}   Let $B$ be the maximum of all the $(k+2)$ $\limsup$ values in Eq. \eqref{g_1_condition}, we discuss the cases $k=0, 1$ and $k\geq 2$.

\paragraph{Case 1: $k=0$} Prove by control both the integral and the quadrature of $h_n(x)$.
\begin{eqnarray}
    &&|h_n(x)|=\left|g(\sqrt{(x-b)^2+\frac{1}{n^4}}) - g(|x-b|)\right| \nonumber\\
    &=& \left| \int_{|x-b|}^{\sqrt{(x-b)^2+\frac{1}{n^4}}} g'(s) ds \right|
    \leq \int_{|x-b|}^{\sqrt{(x-b)^2+\frac{1}{n^4}}} \frac{B}{s^{1-\alpha}} ds \\
    &=& \frac{B}{\alpha} \left[\left((x-b)^2+\frac{1}{n^4}\right)^{\alpha/2} - |x-b|^\alpha \right] \nonumber\\
    &=& \frac{B}{\alpha n^{2\alpha}} \left[\left(t^2+1\right)^{\alpha/2} - |t|^\alpha \right] \triangleq \frac{H_1(t)}{n^{2\alpha}}.
\end{eqnarray}
Here $t=\frac{x-b}{n^2}$. Since $0 < \alpha \leq 1$, $H_1(t)$ is an even function that decreases monotonically on $[0,\infty)$. According to Eq.~\eqref{P_real}, the Gauss quadrature nodes can be uniformly approximated by $x_j = \cos(\frac{4j-1}{4n+2}\pi) + O(\frac{1}{n})$, for $j=1,\cdots,n$. Let $j^*$ denote the index of the quadrature node closest to $b$. For $|j-j^*| \leq 1$ and sufficiently large $n$,
\begin{equation}
    |h_n(x_j)| \leq \frac{B}{\alpha n^{2\alpha}}|H_1(0)| = \frac{B}{\alpha n^{2\alpha}}, \ \ \ \omega_j = \frac{2}{(1-x_i^2)[P'_n(x_i)]^2} \leq \frac{\pi}{2n}\sqrt{1-b^2}.
\end{equation}
For $|j-j^*|\geq 2$, the nodes satisfy $|x_j-b| \geq \sqrt{1-b^2}\frac{\pi}{2n}$, so
\begin{equation}
      |h_n(x_j)| \leq \left|h_n(b+\sqrt{1-b^2}\frac{\pi}{2n})\right| \leq
        \frac{B_1}{n^{\alpha+2}}.
\end{equation}
Since $\sum_{j=1}^n \omega_j = 2$ and the weights are all positive,
\begin{equation}
    \sum_{j=1}^n \omega_j |h_n(x_j)| = \sum_{|j-j^*|\leq 1} \omega_j |h_n(x_j)| + \sum_{|j-j^*|\geq 2} \omega_j |h_n(x_j)| \leq \frac{B_2}{n^{2\alpha+1}}. \label{omega_k=0}
\end{equation}
The same argument can be applied to the integral, 
\begin{equation}
    \int_{-1}^1 |h_n(x)| dx = \int_{-1}^{x_{j^*-2}}+\int_{x_{j^*-2}}^{x_{j^*+2}}+\int_{x_{j^*+2}}^{1} |h_n(x)| dx \leq \frac{B_3}{n^{2\alpha+1}}.
\end{equation}
Although there exists better estimate of $\int_{-1}^1 h_n(x) dx$, it won't improve $R_n[h_n]$. 
\begin{equation}
    |R_n[h_n]| \leq \int_{-1}^1 |h_n(x)| dx + \sum_{j=1}^n \omega_j |h_n(x_j)| \leq \frac{B_4}{n^{2\alpha+1}}. \label{R_h_k=0}
\end{equation}
These $B_1$, $B_2$, $B_3$ and $B_4$ are all constants independent of $x$ and $n$.

\paragraph{Case 2: $k=1$} Similarly, prove by control both the integral and the quadrature of $h_n(x)$. Apply mean value theorem then use Eq. \eqref{g_1_condition},
\begin{eqnarray}
    &&|h_n(x)|=|x-b|\left|g(\sqrt{(x-b)^2+\frac{1}{n^4}}) - g(|x-b|)\right| \\
    &=& |x-b| |g'(s(x))| \left[\sqrt{(x-b)^2+\frac{1}{n^4}} - |x-b| \right] \\
    &\leq& B |x-b|^{\alpha}\left[\sqrt{(x-b)^2+\frac{1}{n^4}} - |x-b| \right] \\
    &=& \frac{B}{ n^{2+2\alpha}} |t|^{\alpha}\left[\sqrt{t^2+1} - |t| \right] \label{h_n_k=1}
\end{eqnarray}

If $0<\alpha<1$, there exists a unique maximum of $H_2(t)$ on $[0,\infty)$, at $t^* = \frac{\alpha}{\sqrt{1-\alpha^2}}$. Following a similar argument as in Eq. \eqref{omega_k=0} and \eqref{R_h_k=0}. For $|j-j^*| \leq 1$ and sufficiently large $n$,
\begin{equation}
    h_n(x_j) \leq \frac{B}{\alpha n^{2+2\alpha}}H_1(t^*) = \frac{B'}{\alpha n^{2+2\alpha}}.
\end{equation}
For $|j-j^*|\geq 2$, and sufficiently large $n$,
\begin{equation}
     |x_j-b| \geq \sqrt{1-b^2}\frac{\pi}{2n} > \frac{t^*}{n^2} \Longrightarrow h_n(x_j) \leq 
        \frac{B'_1}{n^{\alpha+3}}.
\end{equation}
So overall,
\begin{equation}
     \sum_{j=1}^n \omega_j |h_n(x_j)| \leq \frac{B'_2}{n^{3+\alpha}}, \ \ \int_{-1}^1 |h_n(x)| dx \leq \frac{B'_3}{n^{3+\alpha}}, \ \ |R_n[h_n]| \leq \frac{B'_4}{n^{\alpha+3}}. \label{k1alpha<1}
\end{equation}

If $\alpha=1$, the function $H_2(t) =t\left[\sqrt{t^2+1} - |t| \right]$ is monotonic increasing.
\begin{equation}
    \sum_{j=1}^n \omega_j |h_n(x_j)| \leq \frac{B''_2}{n^{4}}, \ \ \int_{-1}^1 |h_n(x)| dx \leq \frac{B''_3}{n^{4}},\ \ |R_n[h_n]| \leq \frac{B''_4}{n^{4}}. \label{k1alpha=1}
\end{equation}
All the above $B$s are also constants independent of $x$ and $n$.

\paragraph{Case 3: $k \geq 2$} We can estimate the $k$-th derivative of $h_n(x)$,
\begin{eqnarray}
    |h_n^{(k)}(x)| &=& \left|\sum_{l=0}^k \frac{(k!)^2 (x-b)^{l}}{(l!)^2(k-l)!} \left[g^{(l)}(\sqrt{(x-b)^2+\frac{1}{n^4}})-g^{(l)}(|x-b|)\right]\right| \\
    &\leq& \sum_{l=0}^k \frac{(k!)^2}{(l!)^2(k-l)!} \left[\sqrt{(x-b)^2+\frac{1}{n^4}}-|x-b|\right] \frac{B}{|x-b|^{1-\alpha}}\label{mvt_g1}\\
    &\leq&  B_5 |x-b|^{\alpha-1}\left[\sqrt{(x-b)^2+\frac{1}{n^4}}-|x-b|\right].
    \label{h_n_k_estimate}
\end{eqnarray}
In Eq. \eqref{mvt_g1}, we used mean value theorem and applied Eq. \eqref{g_1_condition}. The total variation of the $(k-1)$-th derivative is bounded by,
\begin{eqnarray}
    &&V[h_n^{(k-1)}] \leq \int_{-1}^1 |h_n^{(k)}(x)| dx \\
    &\leq& B_5\int_{-1}^1 |x-b|^{\alpha-1}\left[\sqrt{(x-b)^2+\frac{1}{n^4}}-|x-b|\right] dx
    \\
    &\leq& \frac{2B_5}{n^{2+2\alpha}}\int_{0}^{2n^2} t^{\alpha-1}\left[\sqrt{t^2+1}-t\right] dt \\
    &\leq& \begin{cases}
        \displaystyle \frac{B_{6}}{n^{2+2\alpha}},\ \ \ 0<\alpha<1,\\
        \displaystyle \frac{B_{7}\log n}{n^{4}},\ \ \ \alpha=1.
    \end{cases} \label{BV_h_n_k}
\end{eqnarray}
This indicates, if $0<\alpha<1$, then $V[n^{2+2\alpha}h_n^{(k-1)}]$ are uniformly bounded for all $n$. According to Theorem 7.1 in \cite{trefethen2019approximation}, the $m$-th Chebyshev coefficients of, $n^{2+2\alpha}h_n(x)$, denoted as $a_{n,m}$, are uniformly bounded by $\frac{C}{m^k}$ for all $n$ and $m$, with $C$ being a constant independent of $n$ and $m$. According to Theorem 2.3 in \cite{xiang2016improved},
\begin{equation}
    \displaystyle |R_n[h_n]| = \frac{|R_n[n^{2+2\alpha}h_n]|}{n^{2+2\alpha}}  \leq \begin{cases}
        O(\frac{\log n}{n^{k+2\alpha+2}}), \ \ \ \ \text{ if } k = 2\\
        O(\frac{1}{n^{k+2\alpha+2}}), \ \ \ \ \text{ if } k > 2.
    \end{cases}  \label{R_n_h_large}
\end{equation}
Similarly, if $\alpha=1$, then $V[\frac{n^{4}}{\log n}h_n^{(k-1)}]$ are uniformly bounded for all $n$,
\begin{equation}
    \displaystyle |R_n[h_n]| = \frac{|R_n[n^{4}h_n /\log n]|}{n^{4}/\log n}  \leq \begin{cases}
        O(\frac{(\log n)^2}{n^{k+4}}), \ \ \ \ \text{ if } k = 2\\
        O(\frac{\log n}{n^{k+4}}), \ \ \ \ \ \ \ \text{ if } k > 2.
    \end{cases}  \label{R_n_h_large1}
\end{equation}
Combining Eq. \eqref{R_n_h_large}, \eqref{R_n_h_large1}, \eqref{R_h_k=0}, \eqref{k1alpha<1}, and \eqref{k1alpha=1}, completes the proof.
\end{proof}
\begin{proof} {\textit Part (II).} For each $f_n(x)$, there are two branch points, $b \pm i\frac{1}{n^2}$. Eq. \eqref{Gauss_Error_Contour} is used to calculate $R_n[f_n]$. The branch cut and the contour $C_n$ are shown in Fig. \ref{contour}, in positive orientation (counterclockwise). $C_{\epsilon_1} =C_{\epsilon_1} \cup C_{\epsilon_2}$ are two circles around the branch points, with radius $\epsilon$. $C_B$ is the Bernstein ellipse $|z + \sqrt{z^2-1}|=1+\frac{M\log n}{n}$. $C_{l1}$, $C_{l2}$, $C_{l3}$ and $C_{l4}$ are vertical straight lines along the branch cut.
\begin{figure}[H]
    \centering
    \includegraphics[width=0.5\linewidth]{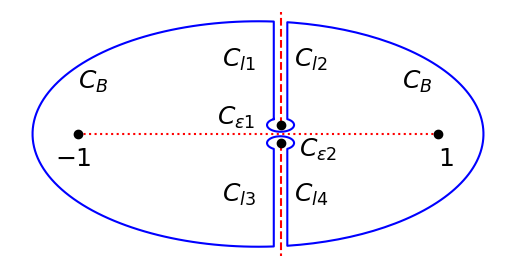}
    \caption{Contour $C_n$ encloses $[-1,1]$ and excludes the branch cut of $f_n(z)$.}
    \label{contour}
\end{figure}
The contribution on $C_\epsilon$ could be made arbitrarily small by letting $\epsilon \to 0$,
\begin{eqnarray}
    &&\lim_{\epsilon \to 0^+} \left|\int_{C_\epsilon} (z-b)^{k} g(\sqrt{(z-b)^2+\frac{1}{n^4}}) \frac{Q_n(z)}{P_n(z)} dz\right| \nonumber \\
    &\leq& \lim_{\epsilon \to 0^+} 2\int_0^{2\pi} \epsilon^{k+1} \left|g(\sqrt{\frac{2i\epsilon}{n^4} + \epsilon^2})\right| \left|\frac{Q_n(z)}{P_n(z)}\right| d\theta = 0. \label{C_epsilon_small}
\end{eqnarray}
In the last step, Eq.~\eqref{g_1_condition} is used. By choosing $M$ sufficiently large, the contribution from $C_B$ becomes a higher-order term, which follows directly from Corollary~\ref{Q_P_ratio_bernstein},
\begin{equation}
    \left|\int_{C_B} (z-b)^{k} g(\sqrt{(z-b)^2+\frac{1}{n^4}}) \frac{Q_n(z)}{P_n(z)} dz\right| \leq O(\frac{1}{n^{2M}}).
\end{equation}
The contributions from $C_{l1}$, $C_{l2}$, $C_{l3}$, and $C_{l4}$ can be parametrized in a similar manner.
\begin{eqnarray}
    &&\int_{C_{l1}} (z-b)^{k} g(\sqrt{(z-b)^2+\frac{1}{n^4}}) \frac{Q_n(z)}{P_n(z)} dz\nonumber\\
    &=& \frac{-i^{k}\pi}{ n^{k+\alpha+1}}\int_{1/n}^{M\log n} y^{k} \frac{  n^\alpha g(-i\frac{\sqrt{y^2-1/n^2}}{n}) }{\exp\left(\frac{2y}{\sin\phi}+i\Psi\right) + 1} dy + R_1\\
    &=& \frac{-i^{k}\pi}{ n^{k+\alpha+1}}\int_{1/n}^{M\log n}  y^{k}\frac{ n^\alpha g(-\frac{iy}{n}) }{\exp\left(\frac{2y}{\sin\phi}+i\Psi\right) + 1} dy + R_1 + R_2\\
    &=& \frac{ -i^{k} \pi}{ n^{k+\alpha+1}}\int_{0}^{M\log n}  y^{k}\frac{ n^\alpha g(-\frac{iy}{n}) }{\exp\left(\frac{2y}{\sin\phi}+i\Psi\right) + 1} dy + R_1 + R_2 +R_3
\end{eqnarray}
Here $R_1$, $R_2$ and $R_3$ are higher order terms, estimated below.
\begin{eqnarray}
    R_1 &\leq & O(\frac{1}{ n^{k+2}})\int_{1/n}^{M\log n}  \left|\frac{ y^{k} g(-i\frac{\sqrt{y^2-1/n^2}}{n}) }{\left[\exp\left(\frac{2y}{\sin\phi}\right)-1\right]^2}  \right| dy \label{R_1_estimate_step1}\\
    &\leq& O(\frac{1}{ n^{k+\alpha+2}})\int_{1/n}^{M\log n}  \frac{ y^{k+\alpha} }{\left[\exp\left(\frac{2y}{\sin\phi}\right)-1\right]^2} dy \label{R_1_estimate_step2}\\
    &\leq& O(\frac{1}{ n^{k+\alpha+2}})\left(\int_{1/n}^{1}  y^{k+\alpha-2} dy +\int_{1}^{M\log n}  \frac{ y^{k+\alpha} }{\left[\exp\left(\frac{2y}{\sin\phi}\right)-1\right]^2} dy\right) \label{R_1_estimate_step3}\\
    &=& \begin{cases}
        O(\frac{1}{n^{2\alpha+2k+1}}), \ \ \ \ \  \text{ if } k+\alpha <1,\\
        O(\frac{\log n}{n^{k+\alpha+2}}), \ \ \ \ \ \ \ \text{ if } k+\alpha =1,\\
        O(\frac{1}{n^{k+\alpha+2}}),\ \ \ \ \ \ \  \text{ if }  k+\alpha >1
    \end{cases} \label{R_1_estimate}
\end{eqnarray}
In Eq. \eqref{R_1_estimate_step1}, use Eq. \eqref{Q_P_ratio}. In Eq. \eqref{R_1_estimate_step2}, use Eq. \eqref{g_1_condition}. In Eq. \eqref{R_1_estimate_step3}, use $\exp(x)-1\geq x$.
\begin{eqnarray}
    R_2& \leq &O(\frac{1}{ n^{k+1}})\int_{1/n}^{M\log n}  \left|y^{k}\frac{  g(-i\frac{\sqrt{y^2-1/n^2}}{n}) - g(-\frac{iy}{n})}{\exp\left(\frac{2y}{\sin\phi}+i\Psi\right) + 1} \right|dy \nonumber\\
    &\leq & O(\frac{1}{ n^{k+\alpha+1}})\int_{1/n}^{\infty}  y^{k}\frac{  y^\alpha - (y^2-1/n^2)^{\alpha/2} }{\exp\left(\frac{2y}{\sin\phi}\right) - 1} dy \label{R_2_estimate_step1}\\
    &\leq& 
        O(\frac{1}{ n^{k+\alpha+3}})\int_{1/n}^{\infty}  \frac{y^{k+\alpha-2}}{\exp\left(\frac{2y}{\sin\phi}\right) - 1} dy,\label{R_2_estimate_step2}\\
    &\leq& \begin{cases}
        O(\frac{1}{n^{2\alpha+2k+1}}), \ \ \ \ \  \text{ if } k+\alpha < 2,\\
        O(\frac{\log n}{n^{5}}), \ \ \ \ \ \ \ \ \ \  \text{ if } k+\alpha = 2,\\
        O(\frac{1}{n^{k+\alpha+3}}),\ \ \ \ \ \ \  \text{ if }  k+\alpha > 2.
    \end{cases}\label{R_2_estimate}
\end{eqnarray}
In Eq. \eqref{R_2_estimate_step1}, use Eq. \eqref{g_1_condition} for the numerator, and $|\exp\left(x+i\theta\right) + 1|\geq |\exp(x)-1|$ for the denominator. In Eq. \eqref{R_2_estimate_step2}, consider the function, $n^2y^2[1-(1-\frac{1}{n^2y^2})^{\alpha/2}]$, which is uniformly bounded on $[\frac{1}{n},\infty)$. In Eq.~\eqref{R_2_estimate}, the dominant contribution arises from the singularity of the integrand near $y=\frac{1}{n}$.
\begin{equation}
    R_3 \leq O(\frac{1}{ n^{k+\alpha+1}})\int_{0}^{1/n} \frac{ y^{k+\alpha} }{\exp\left(\frac{2y}{\sin\phi}\right) - 1} dy \leq O(\frac{1}{ n^{2k+2\alpha+1}}). \label{R_3_estimate}
\end{equation}
In Eq. \eqref{R_3_estimate}, use Eq. \eqref{g_1_condition}, and $e^x-1 \geq x$.

Combing the contributions from $C_{l1}$, $C_{l2}$, $C_{l3}$, $C_{l4}$, $C_\epsilon$ and $C_B$, we obtain the leading order term,
\begin{equation}
    R_n[f_n] \sim \frac{2}{ n^{k+\alpha+1}}\int_{0}^{M \log n} \Re\left(\frac{ i^{k+1}y^{k} n^{\alpha} \left[ g(i\frac{y}{n})\right] }{\exp\left(\frac{2y}{\sin\phi}+i\Psi\right) + 1} \right)dy \leq O(\frac{1}{ n^{k+\alpha+1}}).\label{R_n_f_n_estimate}
    \end{equation}
Eq. \eqref{g_1_condition} implies $\left|n^\alpha g(i\frac{y}{n}) \right| \leq B y^\alpha$, hence the integral in Eq.~\eqref{R_n_f_n_estimate} is bounded. The next order comes from $R_1$, $R_2$, $R_3$ and $R_n[h_n]$.
\end{proof}

Following the same approach, we state another lemma; the proof is provided in the appendix.
\begin{lemma} \label{lemma_g2}
    Consider the $D_b^{k,\alpha}[-1,1]$ function of the form, 
    \begin{equation}
        f(x) = (x-b)^{k+1} g(|x-b|),
    \end{equation}
    with $g(x)$ satisfying Eq. \eqref{g_2_condition}. Let
    \begin{equation}
        f_n(x) = (x-b)^{k+1} g(\sqrt{(x-b)^2+\frac{1}{n^4}}), \ \ \ h_n(x) = f_n(x) - f(x).
    \end{equation}
    Then the remainder term of $n$-point Gauss quadrature on real interval $[-1,1]$  satisfies,
    \begin{itemize}
        \item[(I)] 
        $\displaystyle |R_n[h_n]| \leq  \begin{cases}
        \displaystyle O(\frac{1}{n^{k+\alpha+2}}), \ \ \ \ \text{ if } k \geq 1,\\
        \displaystyle O(\frac{1}{n^{2\alpha+1}}), \ \ \ \ \ \ \text{ if } k=0.
    \end{cases}$
    \item[(II)] Let $b=\cos\phi$, $\Psi =(2n+1)\phi-\frac{\pi}{2}$. Then
    \begin{eqnarray}
            R_n[f] &=& \frac{2}{ n^{k+\alpha+1}}\int_{0}^{M\log n}  \Re\left(\frac{ i^{k+2}y^{k+1} n^{\alpha-1} \left[ g(i\frac{y}{n})\right] }{\exp\left(\frac{2y}{\sin\phi}+i\Psi\right) + 1}\right) dy \nonumber\\
        &&+\begin{cases}
        O(\frac{1}{n^{2\alpha+2k+1}}), \ \ \ \ \  \text{ if } k+\alpha <1,\\
        O(\frac{\log n}{n^{k+\alpha+2}}), \ \ \ \ \ \ \ \text{ if } k+\alpha =1,\\
        O(\frac{1}{n^{k+\alpha+2}}),\ \ \ \ \ \ \  \text{ if }  k+\alpha >1.
    \end{cases}
    \end{eqnarray}
    Here $\left[g(iy)\right] = g(iy)-g(-iy)$ represents the jump across the branch cut. $M$ is a fixed number that chosen to be sufficiently large.
    \end{itemize}
\end{lemma}

\subsection{Proof of Theorem \ref{main_theorem}} For any function $f(x) \in D_b^{k,\alpha}[-1,1]$, Lemma~\ref{f_decomposition} implies that it can be decomposed as
\begin{equation}
    f(x) = P_k(x) + (x-b)^k g_1(|x-b|) + (x-b)^{k+1} g_2(|x-b|).
\end{equation}
There is no quadrature error for the polynomial part $P_k(x)$ when $n \geq k/2$. 
Hence, the leading-order remainder is given by the sum of the leading terms in Lemmas~\ref{lemma_g1} and~\ref{lemma_g2}. 
We also note that the jump of polynomials across the branch cut is zero.
\begin{eqnarray}
    &&R_n[f] = R_n[(x-b)^k g_1(|x-b|)] + R_n[(x-b)^{k+1} g_2(|x-b|)] \\
    &\sim& \frac{2}{ n^{k+\alpha+1}}\int_{0}^{M\log n}  \Re\left(\frac{ i^{k+1}y^{k} n^{\alpha} \left[ g_1(i\frac{y}{n})\right] }{\exp\left(\frac{2y}{\sin\phi}+i\Psi\right) + 1}\right) dy  \nonumber\\
    &&+\frac{2}{ n^{k+\alpha+1}}\int_{0}^{M\log n}  \Re\left(\frac{ i^{k+2}y^{k+1} n^{\alpha-1} \left[ g_2(i\frac{y}{n})\right] }{\exp\left(\frac{2y}{\sin\phi}+i\Psi\right) + 1}\right) dy  \\
    &=&\frac{2}{n}\int_{0}^{M\log n}  \Re \left(\frac{  i\left[ f(b+i\frac{y}{n})\right] }{\exp\left(\frac{2y}{\sin\phi}+i\Psi\right) + 1} \right) dy\\
    &=&\frac{1}{n} \int_{0}^{M\log n} \Re\left( \left[ f(b+i\frac{y}{n})\right] \frac{  i\exp\left(\frac{-2y}{\sin\phi}\right) + i\cos \Psi + \sin \Psi  }{\cosh\left(\frac{2y}{\sin\phi}\right) + \cos \Psi} \right)dy.
\end{eqnarray}
The proof is then complete. Remark \ref{main_order} is a consequence of Eq. \eqref{R_n_f_n_estimate}.

\subsection{Application to Power Singularity\label{power_function}} As an example, consider $-1<b<1$, $\alpha \in (-1,0) \cup (0,1]$, integer $k \geq 0$ and $k + \alpha >0$, the integrand
\begin{equation}
    f(x) = (x-b)^{k} |x-b|^{\alpha}. \label{power_singular_function}
\end{equation}
When $\alpha>0$, then $f(x) \in D^{k,\alpha}_b[-1,1]$, corresponds to Lemma \ref{lemma_g1}. When $\alpha<0$, then $f(x) \in D^{k-1,\alpha+1}_b[-1,1]$, corresponds to Lemma \ref{lemma_g2}. For $y>0$, the jump across the branch cut $(b-i\infty,b+i\infty)$ is given by,
\begin{equation}
    f(b^++i\frac{y}{n}) = i^k \frac{y^{k+\alpha}}{n^{k+\alpha}} e^{i\alpha \pi/2},\ \ \  f(b^-+i\frac{y}{n}) = i^k \frac{y^{k+\alpha}}{n^{k+\alpha}} e^{-i\alpha \pi/2},
\end{equation}
\begin{equation}
    [f(b+i\frac{y}{n})] = 2i^{k+1}\sin \frac{\alpha \pi} {2} \frac{y^{k+\alpha}}{n^{k+\alpha}}.
\end{equation}
The Gauss Quadrature remainder by Eq. \eqref{leading_order_theorem} has leading term,
\begin{eqnarray}
        R_n[f] \sim \frac{\sin(\frac{\alpha\pi}{2})}{ n^{k+\alpha+1}}\int_{0}^{\infty} \Re \left(y^{k+\alpha} i^{k+1}\frac{  i\exp\left(\frac{-2y}{\sin\phi}\right) + i\cos \Psi + \sin \Psi  }{\cosh\left(\frac{2y}{\sin\phi}\right) + \cos \Psi} \right)dy.
    \end{eqnarray}
Notice we have replaced $M\log n$ with $\infty$, because the above integral on $[0,\infty)$ converges with any given $n$. Further simplification requires discussion the parity of $k$.
\paragraph{\bf{Case 1: $k\equiv 0\mp 1\pmod 4$}} Here we write $0$ and $2$ as $1 \mp 1$, allowing the two cases to be treated in a unified manner.
\begin{eqnarray}
        R_n[f] = \frac{\mp \sin(\frac{\alpha\pi}{2})}{ n^{k+\alpha+1}}\int_{0}^{\infty} y^{k+\alpha} \frac{  \exp\left(\frac{-2y}{\sin\phi}\right) + \cos \Psi }{\cosh\left(\frac{2y}{\sin\phi}\right) + \cos \Psi} dy +o(\frac{1}{n^{k+\alpha+1}}).
    \end{eqnarray}
Consider $\Psi = (2n+1)\phi - \frac{\pi}{2}$ and $\cos \Psi \in [-1,1]$,
\begin{equation}
    \limsup_{n \to \infty} \frac{  \exp\left(\frac{-2y}{\sin\phi}\right) + \cos \Psi }{\cosh\left(\frac{2y}{\sin\phi}\right) + \cos \Psi} \leq   \frac{  \exp\left(\frac{-2y}{\sin\phi}\right) +1 }{\cosh\left(\frac{2y}{\sin\phi}\right) +1} = \frac{  2 }{\exp\left(\frac{2y}{\sin\phi}\right) +1}.
\end{equation}
\begin{equation}
    \liminf_{n \to \infty} \frac{  \exp\left(\frac{-2y}{\sin\phi}\right) + \cos \Psi }{\cosh\left(\frac{2y}{\sin\phi}\right) + \cos \Psi} \geq   \frac{  \exp\left(\frac{-2y}{\sin\phi}\right) -1 }{\cosh\left(\frac{2y}{\sin\phi}\right) -1} = \frac{  -2 }{\exp\left(\frac{2y}{\sin\phi}\right) -1}.
\end{equation}
The equality can be attained when $\cos\Psi = \pm 1$ respectively. According to Fatou's lemma and inverse Fatou's lemma,
\begin{eqnarray}
    && \limsup_{n \to \infty} \pm n^{k+\alpha+1}R_n[f] \leq 2 \sin \frac{\alpha \pi}{2}  \int_0^\infty \frac{y^{k+\alpha} }{\exp\left(\frac{2y}{\sin \phi}\right) -1} dy \\
    &=&  \sin\frac{\alpha \pi}{2} (\sin \phi)^{k+\alpha+1}  \frac{\Gamma(k+\alpha+1)}{2^{k+\alpha-1}} \zeta(k+\alpha+1). \label{4l02_u}
\end{eqnarray}
\begin{eqnarray}
    &&\liminf_{n \to \infty} \pm n^{k+\alpha+1}R_n[f] \geq -2 \sin \frac{\alpha \pi}{2}  \int_0^\infty \frac{y^{k+\alpha} }{\exp\left(\frac{2y}{\sin \phi}\right) +1} dy \\
    &=&  -\sin\frac{\alpha \pi}{2} (\sin \phi)^{k+\alpha+1}  \frac{\Gamma(k+\alpha+1)}{2^{k+\alpha-1}}(1-\frac{1}{2^{k+\alpha}}) \zeta(k+\alpha+1). \label{4l02_l}
\end{eqnarray}
Here $\Gamma$ is the gamma function, and $\zeta$ is the Riemann zeta function.
\begin{remark}
    For $k\equiv 0,2\pmod 4$ the limit inferior and limit superior of the leading coefficient have opposite signs. Since the leading coefficient depends continuously on $\cos\Psi$, the intermediate value theorem implies the existence of $\Psi_0$ (equivalently of $\cos\Psi_0$) for which the leading term appearing in Theorem~\ref{main_theorem} is zero. Consequently, the $n$-point Gauss quadrature error decays at a rate strictly faster than $n^{-k-\alpha-1}$. In practice, when applying Gauss Quadrature to functions with singularity of the form \eqref{power_singular_function}, and $k\equiv 0,2\pmod 4$, the number of quadrature points $n$ should be chosen such that $\cos((2n+1)\phi-\frac{\pi}{2}) \approx \cos \Psi_0$, which will be superior comparing with neighboring $n$. Such $\cos \Psi_0$ is implicitly given by,
    \begin{equation}
        \int_{0}^{\infty} x^{k+\alpha} \frac{  e^{-x} + \cos \Psi_0 }{\cosh(x) + \cos \Psi_0} dx = 0.
    \end{equation}
    However, there is no simple expression for $\Psi_0$. Given the fact that, in this case, the limit inferior and limit superior of the leading coefficient are attained when $\cos \Psi =\pm 1$. A crude estimate would be $\cos \Psi_0 =0$, i.e., $\cos((2n+1)\phi) = \pm 1$. \label{good_02}
\end{remark}
\paragraph{\bf{Case 2: $k\equiv \pm 1\pmod 4$}} Similarly,
\begin{equation}
        R_n[f] =\frac{\mp \sin(\frac{\alpha\pi}{2})}{ n^{k+\alpha+1}}  \int_{0}^{\infty}  \frac{  y^{k+\alpha} \sin \Psi }{\cosh\left(\frac{2y}{\sin\phi}\right) + \cos \Psi} dy. \label{pm1_error_integral}
\end{equation}
There exists uniform bound of the integrand,
\begin{equation}
    -\frac{1}{\sinh\left(\frac{2y}{\sin\phi}\right)} \leq \frac{ \sin \Psi }{\cosh\left(\frac{2y}{\sin\phi}\right)+\cos \Psi} \leq \frac{1}{\sinh\left(\frac{2y}{\sin\phi}\right)}.
\end{equation}
Both equal sign are attained when $\cos\Psi = \frac{1}{\cosh\left(\frac{2y}{\sin\phi}\right)}$ with opposite $\sin \Psi$. The limit inferior and limit superior of the leading coefficient can be estimated,
\begin{eqnarray}
    &&\liminf_{n \to \infty} n^{k+\alpha+1}R_n[f] \geq -2 \sin \frac{\alpha \pi}{2}  \int_0^\infty \frac{y^{k+\alpha} }{\sinh\left(\frac{2y}{\sin\phi}\right)} dy \nonumber\\
    &\geq& -\sin\frac{\alpha \pi}{2} (\sin \phi)^{k+\alpha+1}  \frac{\Gamma(k+\alpha+1)}{2^{k+\alpha-1}}(1-\frac{1}{2^{\alpha+1}}) \zeta(k+\alpha+1). \label{4lpm1_l}
\end{eqnarray}
\begin{eqnarray}
    &&\limsup_{n \to \infty} n^{k+\alpha+1}R_n[f] \leq 2 \sin \frac{\alpha \pi}{2}  \int_0^\infty \frac{y^{k+\alpha} }{\sinh\left(\frac{2y}{\sin\phi}\right)} dy \nonumber\\
    &\leq& \sin\frac{\alpha \pi}{2} (\sin \phi)^{k+\alpha+1}  \frac{\Gamma(k+\alpha+1)}{2^{k+\alpha-1}}(1-\frac{1}{2^{\alpha+1}}) \zeta(k+\alpha+1). \label{4lpm1_u}
\end{eqnarray}
Notice, these are strict bounds that cannot be attained. A better estimate could be found by evaluating the integral in Eq. \eqref{pm1_error_integral} as a function of $\Psi$.
\begin{remark}
    In this case of $k\equiv \pm 1\pmod 4$, when $\sin \Psi = \cos((2n+1)\phi) =0$, the leading order term of $R_n[f]$ given by Theorem \ref{main_theorem} becomes 0. In practice, when applying Gauss Quadrature to functions with singularity of the form \eqref{power_singular_function}, and $k\equiv \pm 1\pmod 4$, the number of quadrature points $n$ should be chosen such that $\cos((2n+1)\phi) \approx 0$, which will be superior comparing with neighboring $n$. This behavior is opposite to the case of Remark \ref{good_02}. \label{good_pm1}
\end{remark}

\subsection{Application to Logarithm Singularity}\label{logarithm_function} Another example is the functions with logarithm singularity, at $b = \cos \phi \in (-1,1)$.
\begin{equation}
    f(x) = (x-b)^k |x-b|^\beta\log|x-b|,
\end{equation}
\subsubsection{$0<\beta\leq 1$, $k \geq 0$} In this case, $f(x) \in D_b^{k,\alpha}[-1,1]$, for any $0<\alpha<\beta$. \newline Or $-1<\beta\leq 0$, $k\geq 1$, $f(x) \in D_b^{k-1,\alpha}[-1,1]$, for any $0<\alpha<\beta+1$.

For $y>0$, the jump across the branch cut $(b-i\infty,b+i\infty)$ is given by,
\begin{equation}
    f(b^++i\frac{y}{n}) = (\frac{iy}{n})^k (e^{i\pi}\frac{y^2}{n^2})^{\beta/2} \frac{1}{2} \log (e^{i\pi}\frac{y^2}{n^2}) = i^k e^{i\pi\beta/2} \frac{y^{k+\beta}}{n^{k+\beta}} (\log \frac{y}{n} + \frac{i\pi}{2}),
\end{equation}
\begin{equation}
    f(b^-+i\frac{y}{n}) = (\frac{iy}{n})^k (e^{-i\pi}\frac{y^2}{n^2})^{\beta/2} \frac{1}{2} \log (e^{-i\pi}\frac{y^2}{n^2}) = i^k e^{-i\pi\beta/2} \frac{y^{k+\beta}}{n^{k+\beta}} (\log \frac{y}{n}- \frac{i\pi}{2}),
\end{equation}
\begin{equation}
    [f(b+i\frac{y}{n})] = i^{k+1} \frac{y^{k+\beta}}{n^{k+\beta}} \left[2 \sin\frac{\beta \pi}{2} \log \frac{y}{n} + \pi \sin\frac{(\beta+1) \pi}{2}  \right].
\end{equation}
The Gauss Quadrature remainder by Theorem \ref{main_theorem} has leading term,
\begin{eqnarray}
    R_n[f] &\sim& \frac{1}{ n^{k+\beta+1}}\int_{0}^{\infty} \Re \left(i^{k+1} y^{k+\beta} \left[2 \sin \frac{\beta \pi}{2}\log \frac{y}{n} + \pi \sin \frac{(\beta+1)\pi}{2}\right] \right.\nonumber\\ && \left.\times \frac{  i\exp\left(\frac{-2y}{\sin\phi}\right) + i\cos \Psi + \sin \Psi  }{\cosh\left(\frac{2y}{\sin\phi}\right) + \cos \Psi}\right)dy= O(\frac{\log n}{n^{k+\beta+1}}). \label{betaneq0}
\end{eqnarray}

\subsubsection{$\beta = 0$, $k\geq 1$} In this case, $f(x) \in D_b^{k-1,\alpha}$, for any $0<\alpha<1$.
\begin{equation}
        R_n[f] = \frac{\pi}{ n^{k+1}}\int_{0}^{\infty}\Re\left( y^{k} i^{k+1} \frac{  i\exp\left(\frac{-2y}{\sin\phi}\right) + i\cos \Psi + \sin \Psi  }{\cosh\left(\frac{2y}{\sin\phi}\right) + \cos \Psi} \right)dy = O(\frac{1}{n^{k+1}}). \label{beta =0}
    \end{equation}
Similar discussions about the parity of $k$ and Remark \ref{good_02} and \ref{good_pm1} hold true.

\begin{figure}[htbp]
    \centering
    \includegraphics[width=0.45\linewidth]{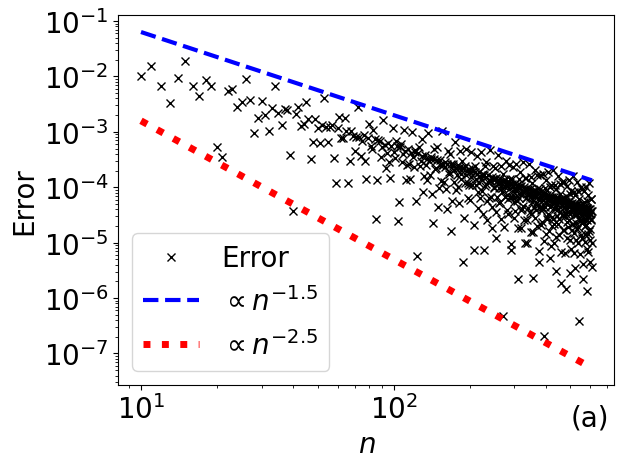}\includegraphics[width=0.45\linewidth]{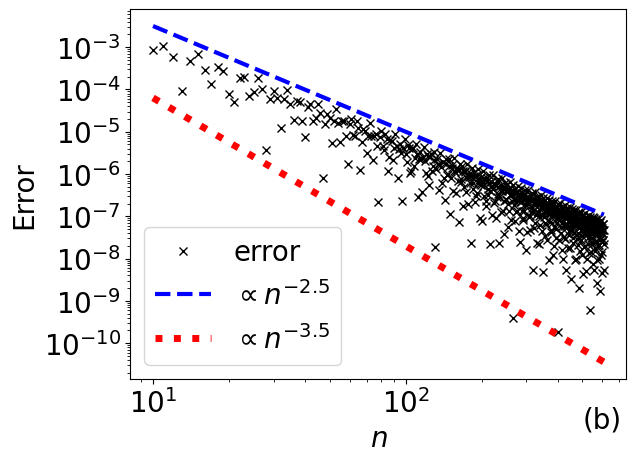}
    \includegraphics[width=0.45\linewidth]{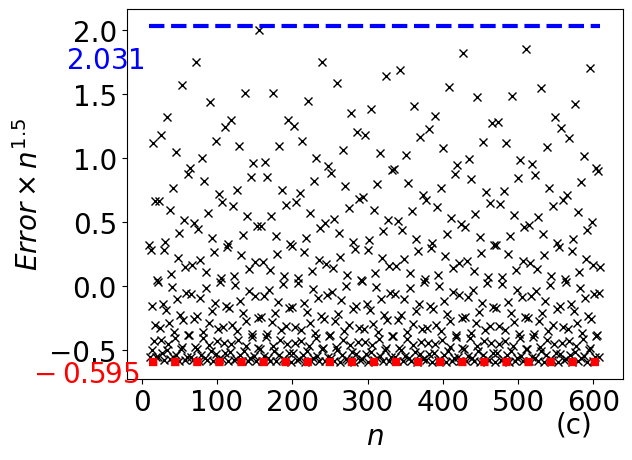}
    \includegraphics[width=0.45\linewidth]{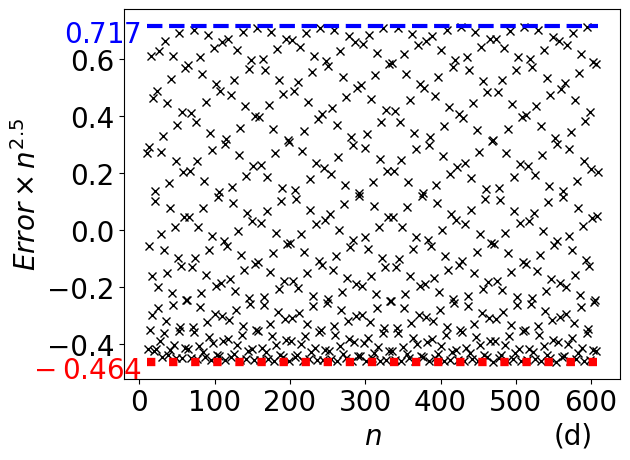}
    \includegraphics[width=0.45\linewidth]{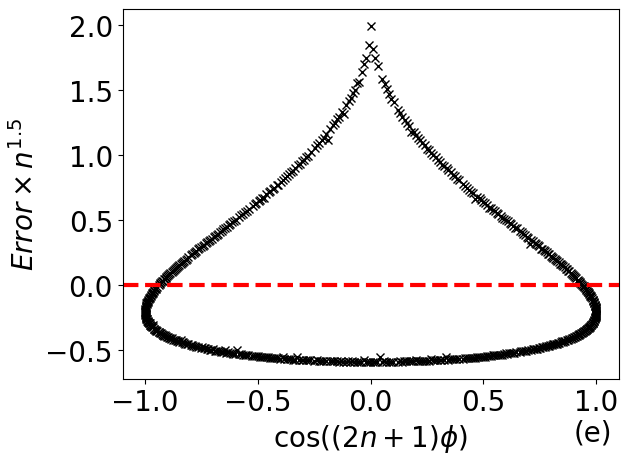}
    \includegraphics[width=0.45\linewidth]{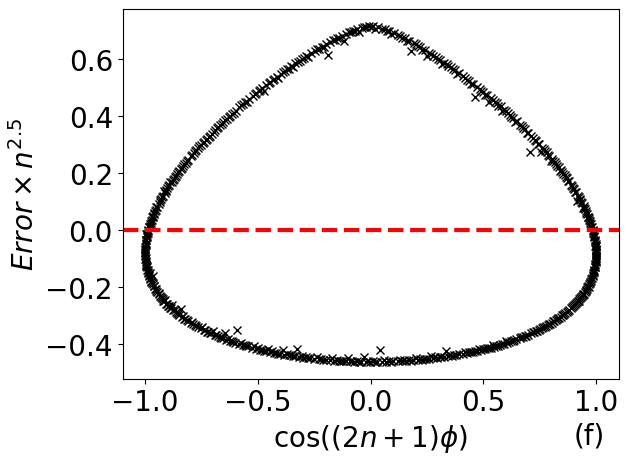}
    \caption{Error in applying $n$-point Gauss Legendre Quadrature to approximate $\int_{-1}^1 |x-0.4|^\alpha dx$, with $n$ being all the integers from $10$ to $600$. Panel (a,c,e): $\alpha = 0.5$; Panel (b,d,f): $\alpha = 1.5$. The leading coefficient in Panel (c,d) are computed through Eq. \eqref{4l02_l} and \eqref{4l02_u}.}
    \label{b_04_k_0}
\end{figure}

\section{Numerical Experiments} \label{sec:numerics}
\subsection{Example 1} In this example, we consider the following power singularity,
\begin{equation}
    f(x) = |x-0.4|^\alpha,
\end{equation}
with $\alpha = 0.5$ or $1.5$. The numerical error analysis is shown in Fig. \ref{b_04_k_0}.

Fig. \ref{b_04_k_0} Panel (a) and (b) plot the absolute error as a function of $n$. All the data points marked by $\times$ are below the dashed line representing the worst case error asymptotes $O(\frac{1}{n^{\alpha+1}})$. However, the error is not monotonic decreasing as a function of $n$. There exist certain values of $n$ such that the error converges faster than the worst case, behaves like $O(\frac{1}{n^{\alpha+2}})$, represented by the dotted line.

In Panel (c) and (d), we consider the leading coefficient by multiplying $n^{k+\alpha+1}$ to the error, without changing its sign. The maximum and minimum leading coefficient are computed through Eq. \eqref{4l02_l} and \eqref{4l02_u}, with the dashed line represents the limit superior and the dotted line represents the limit inferior. 

Panel (e) and (f) illustrate the periodic pattern. Besides the order of convergence being $O(\frac{1}{n^{\alpha+1}})$, the leading coefficient is a function of $\cos((2n+1)\phi)$. When $\cos((2n+1)\phi)=0$, equivalently $\sin((2n+1)\phi)=\pm 1$, the limit superior and limit inferior are attained. There are intersections with the horizontal dashed line, representing the leading coefficient becomes zero, so that the convergence is higher order. This is consistent with the discussion in Remark \ref{good_02}.

\subsection{Example 2} In this example, we consider the following power singularity,
\begin{equation}
    f(x) = (x-0.4)^k |x-0.4|.
\end{equation}
For the case of $k=0$ and $2$, the limit superior and limit inferior of the leading coefficient, calculated from Eq. \eqref{4l02_l} and \eqref{4l02_u}, are exact, as shown in Panel (a) and (c). The maximum and minimum are attained at $\cos((2n+1)\phi)=0$. The leading coefficient becomes close to 0, when $\cos((2n+1)\phi)$ is around $\pm1$, as discussed in Remark \ref{good_02}

For the case of $k=1$ and $3$ in Panel (b) and (d), these extreme values given by Eq. \eqref{4lpm1_l} and \eqref{4lpm1_u} cannot be attained, but they serve as good upper and lower bounds for the leading coefficient. The leading coefficient becomes close to 0, when $\cos((2n+1)\phi) =0$, as discussed in Remark \ref{good_pm1}.
\begin{figure}[htbp]
    \centering
    \includegraphics[width=0.475\linewidth]{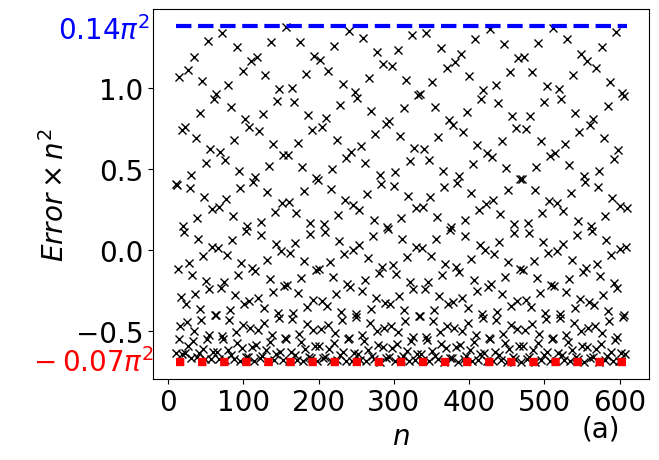}
    \includegraphics[width=0.45\linewidth]{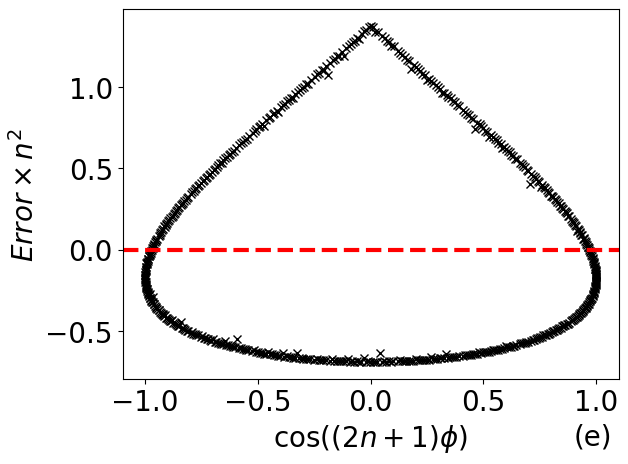}
    \includegraphics[width=0.45\linewidth]{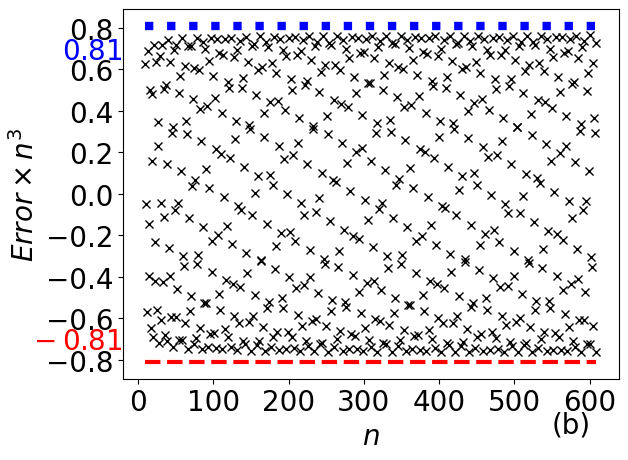}
\includegraphics[width=0.45\linewidth]{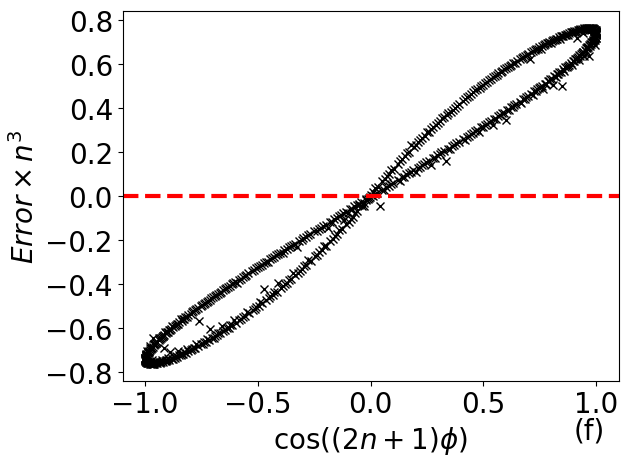}
    \includegraphics[width=0.475\linewidth]{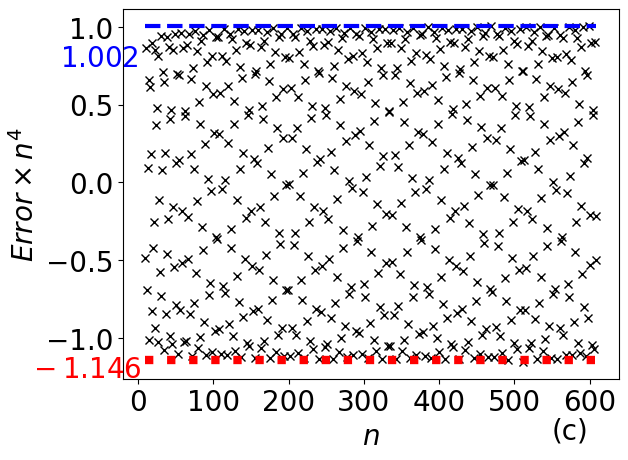}
\includegraphics[width=0.45\linewidth]{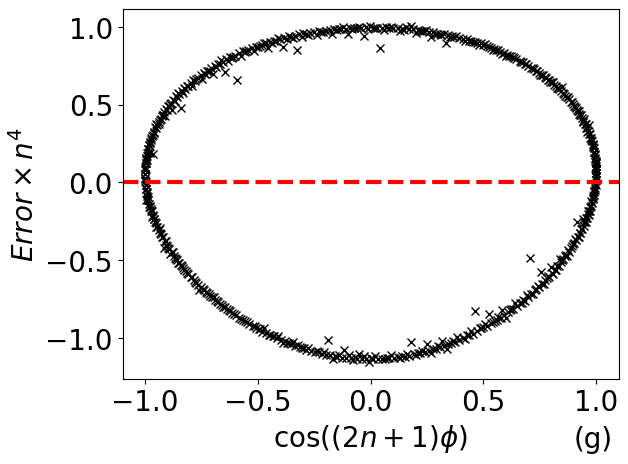}
    \includegraphics[width=0.475\linewidth]{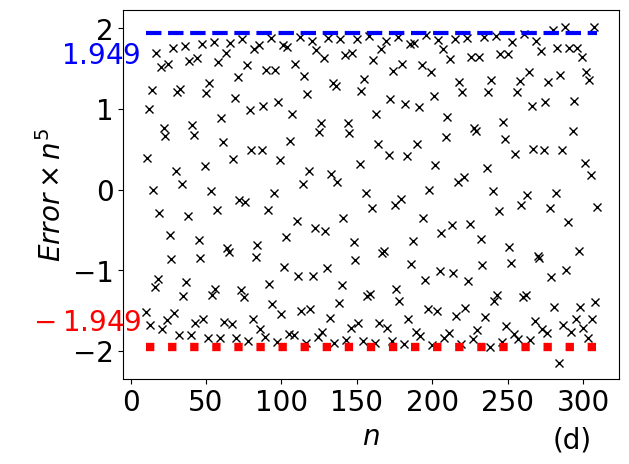}
    \includegraphics[width=0.45\linewidth]{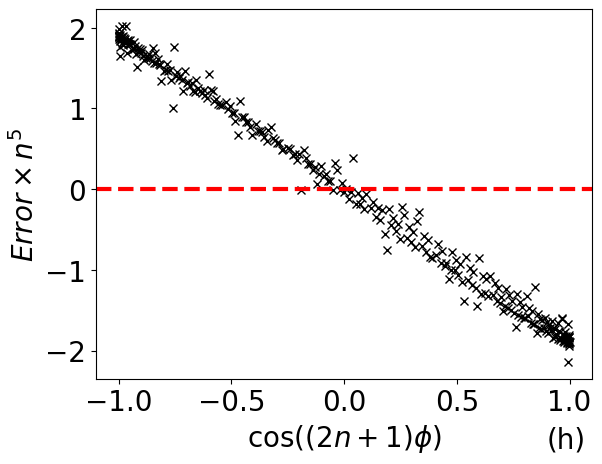}
    \caption{Error in applying $n$-point Gauss Legendre Quadrature to  $\int_{-1}^1 (x-0.4)^k |x-0.4| dx$, with $n$ being all the integers from $10$ to $600$. Panel (a): $k=0$; (b): $k=1$; (c): $k = 2$; (d): $k=3$. For panel (d), the error is at machine accuracy when $n$ is about 300.}
    \label{b_04_a_1}
\end{figure}

\subsection{Example 3} In this example, we consider a different $b=\cos\frac{\pi}{6}$.
\begin{equation}
    f(x) = (x - b)|x-b|^{1/2}.
\end{equation}
In this case, the order of the error is more clear as shown in Fig. \ref{pi_6}. This is due to the fact that, there are only 4 possible values of $\cos((2n+1)\phi)$, when $\phi = \frac{\pi}{6}$. So there are 6 separate curves as shown in Panel (a), which are symmetric in absolute value, and 3 separate straight lines in Panel (b). The 2 straight lines that parallel to $\frac{1}{n^{2.5}}$ are very close, with slight difference in the leading coefficient. The other one, parallel to $\frac{1}{n^{3.5}}$, corresponds to the case $\cos((2n+1)\phi) = \pm 1$.
\begin{figure}[H]
    \centering
    \includegraphics[width=0.45\linewidth]{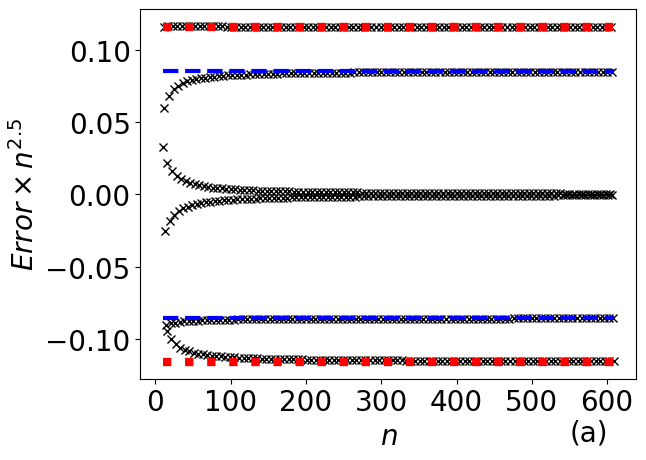}
    \includegraphics[width=0.45\linewidth]{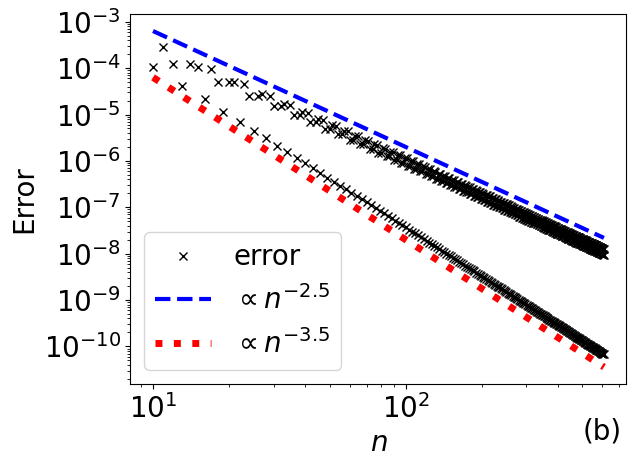}
    \caption{Error in applying $n$-point Gauss Legendre Quadrature to  $\int_{-1}^1 (x-b) |x-b|^{1/2} dx$, with $n$ being all the integers from $10$ to $600$ and $b = \cos(\frac{\pi}{6})$.}
    \label{pi_6}
\end{figure}

\subsection{Example 4} In this example, we consider the following two functions with logarithm singularity,
\begin{equation}
    f_1(x) = |x-0.4| \log|x-0.4|, \ \ f_2(x) = (x-0.4) \log|x-0.4|. \label{log_example}
\end{equation}

$f_1(x)$ corresponds to the case $k=0$, $\beta=1$ in Eq. \eqref{betaneq0}, thus
\begin{equation}
    R_n[f_1] = \frac{2 }{ n^{2}}\int_{0}^{\infty}   \frac{  \exp\left(\frac{-2y}{\sin\phi}\right) + \cos \Psi}{\cosh\left(\frac{2y}{\sin\phi}\right) + \cos \Psi} y\log\frac{n}{y} dy. \label{f_2_error}
\end{equation}
The asymptotic bounds are given by $\cos\Psi = \pm 1$, as the leading order of the integral comes from $y <n$, thus $\log \frac{n}{y}$ can be treated as a positive function.
\begin{equation}
    \limsup_{n\to\infty} n^2 R_n[f_1] \leq 2\int_{0}^{\infty}   \frac{  \exp\left(\frac{-2y}{\sin\phi}\right) + 1}{\cosh\left(\frac{2y}{\sin\phi}\right) + 1} y\log\frac{n}{y} dy \approx 0.691 \log n + 0.162.
\end{equation}
\begin{equation}
    \liminf_{n\to\infty} n^2 R_n[f_1] \geq 2\int_{0}^{\infty}   \frac{  \exp\left(\frac{-2y}{\sin\phi}\right) - 1}{\cosh\left(\frac{2y}{\sin\phi}\right) - 1} y\log\frac{n}{y} dy \approx -1.382 \log n - 1.282.
\end{equation}
The above numbers are evaluated numerically from the integrals. We should notice, these bounds matches exactly as shown in Fig. \ref{log_discontinuous}.

$f_2(x)$ corresponds to the case $k=1$, $\beta=0$ in Eq. \eqref{beta =0}, thus
\begin{equation}
        R_n[f_2] = \frac{-\pi}{ n^{2}}\int_{0}^{\infty}  \frac{ y\sin \Psi  }{\cosh\left(\frac{2y}{\sin\phi}\right) + \cos \Psi} dy. \label{f_1_error}
\end{equation}
The bounds can be estimated using $\cos\Psi = \frac{1}{\cosh\left(\frac{2y}{\sin\phi}\right)}$.
\begin{equation}
    \limsup_{n\to\infty} \left|n^2 R_n[f_2]\right| \leq \pi \int_{0}^{\infty}   \frac{  y}{\sinh\left(\frac{2y}{\sin\phi}\right)} dy \approx 1.628.
\end{equation}
These bounds cannot be attained, as shown in Fig. \ref{log_discontinuous}. It is interesting to notice, although $f_1$ and $f_2$ are both continuous, with logarithmic singular derivative at $x=0.4$, their Gauss quadrature error decays in a different way. 

\begin{figure}[!htbp] \label{log_numeric}
    \centering
    \includegraphics[width=0.45\linewidth]{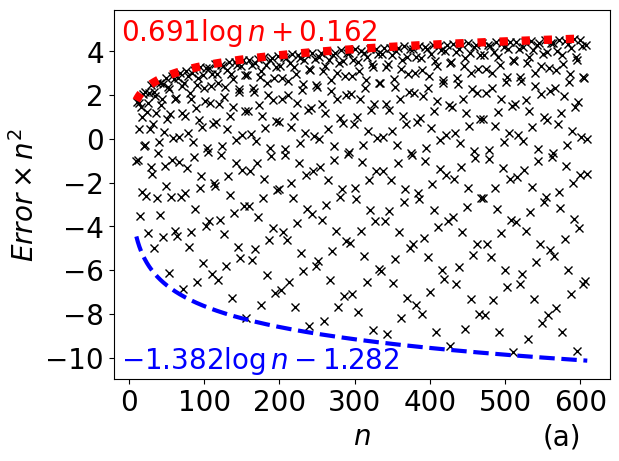}
    \includegraphics[width=0.45\linewidth]{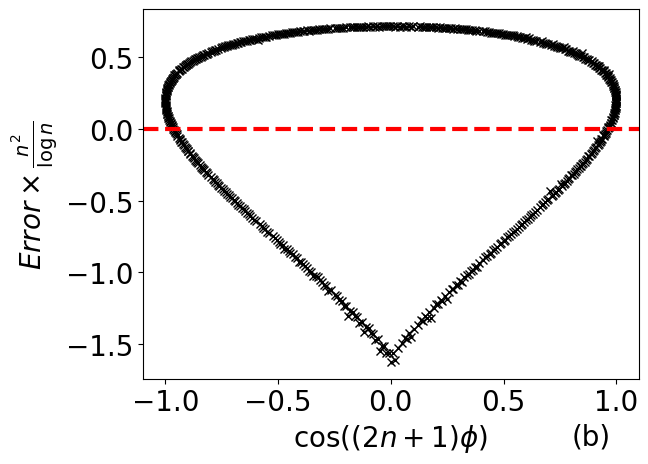}
    \includegraphics[width=0.45\linewidth]{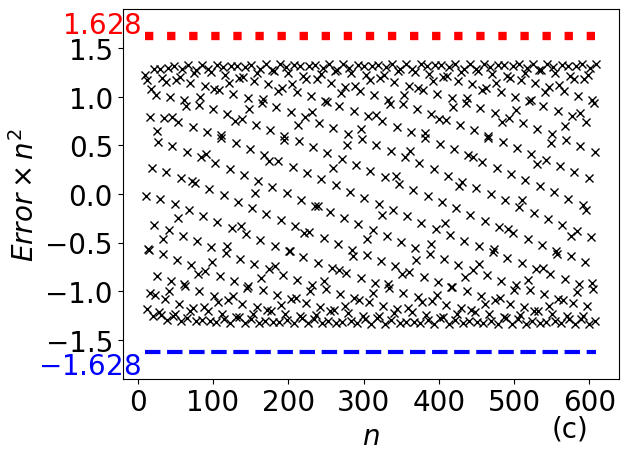}
    \includegraphics[width=0.45\linewidth]{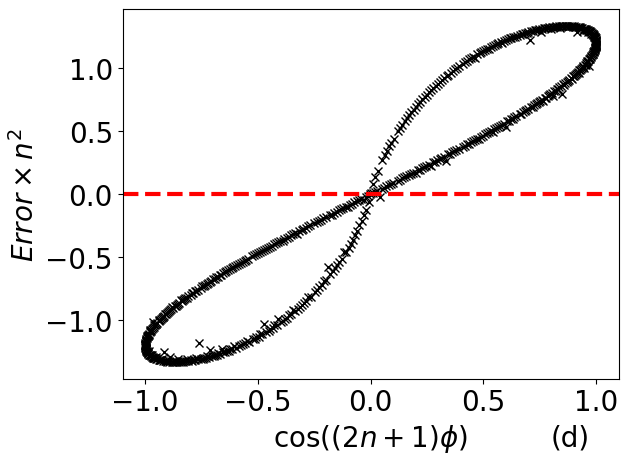}
    
    \caption{Error in applying $n$-point Gauss Legendre Quadrature to functions with logarithm singularity, given by Eq. \eqref{log_example}, with $n$ being all the integers from $10$ to $600$. Panel (a),(b): $f_1$. Panel (c),(d): $f_2$.}
    \label{log_discontinuous}
\end{figure}

\subsection{Example 5}\label{example5} In this example, we consider a more general case.
\begin{equation}
    f(x) = e^{-(x-0.4)^2} |x-0.4|,
\end{equation}
whose derivative is discontinuous at $x=0.4$.
\begin{figure}[!htbp]
    \centering
    \includegraphics[width=0.42\linewidth]{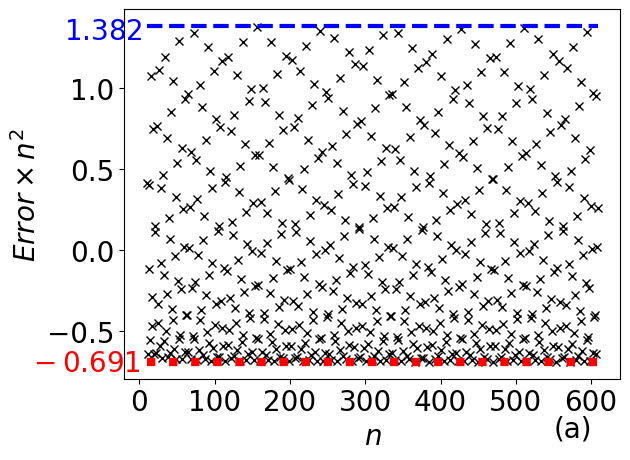}
    \includegraphics[width=0.42\linewidth]{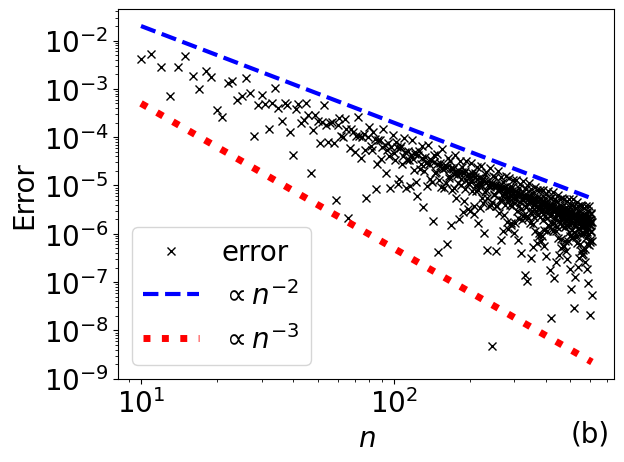}
    \includegraphics[width=0.42\linewidth]{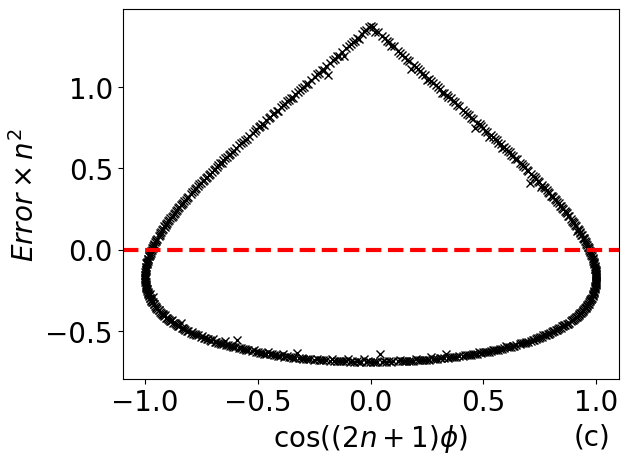}
    \includegraphics[width=0.42\linewidth]{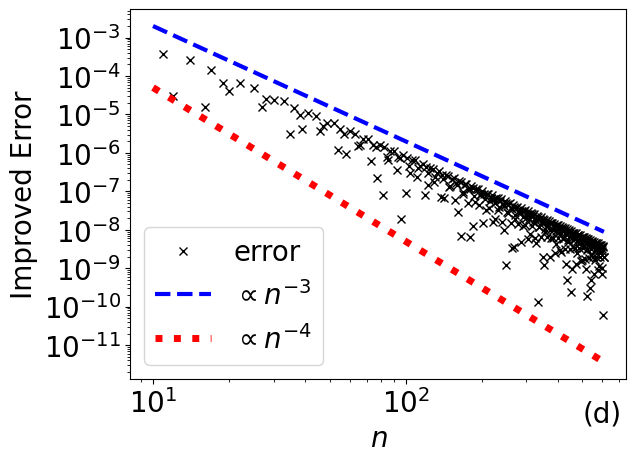}
    
    \caption{Error in applying $n$-point Gauss Legendre Quadrature to $\displaystyle \int_{-1}^1 e^{-(x-0.4)^2} |x-0.4| dx$, with $n$ being all the integers from $10$ to $600$. The bounds in Panel (a) are computed from Eq. \eqref{example_5_error}.}
    \label{general_example}
\end{figure}
Notice the jump 
\begin{equation}
    [f(0.4+\frac{iy}{n})] = 2\frac{iy}{n}\exp\left(\frac{y^2}{n^2}\right).
\end{equation}
The leading order of Gauss quadrature error given by Theorem \ref{main_theorem} reads,
\begin{equation}
    \frac{-2}{ n^2} Re\int_{0}^{M \log n}  y\exp\left(\frac{y^2}{n^2}\right) \frac{  \exp\left(\frac{-2y}{\sin\phi}\right) + \cos \Psi - i \sin \Psi  }{\cosh\left(\frac{2y}{\sin\phi}\right) + \cos \Psi} dy.
\end{equation}
For the leading order contribution, we can use Taylor's expansion, so that
\begin{equation}
    R_n[f] \sim \frac{-2}{ n^2} \int_{0}^{\infty}  y \frac{  \exp\left(\frac{-2y}{\sin\phi}\right) + \cos \Psi }{\cosh\left(\frac{2y}{\sin\phi}\right) + \cos \Psi} dy. \label{example_5_error}
\end{equation}
The bounds of the leading coefficients in Panel (a) are attained when $\cos\Psi = \pm 1$. Panel (d) considers Eq. \eqref{example_5_error} as a correction to the Gauss quadrature, and plots the improved error, $\int_{-1}^1 f(x) dx - \sum_{j=1}^n \omega_j f(x_j) + \frac{2}{ n^2} \int_{0}^{\infty}  y \frac{  \exp\left(\frac{-2y}{\sin\phi}\right) + \cos \Psi }{\cosh\left(\frac{2y}{\sin\phi}\right) + \cos \Psi} dy$.
It clearly shows the worst case convergence order improves from $O(\frac{1}{n^2})$ to $O(\frac{1}{n^3})$.

\section{Appendix}

\subsection{Proof of Lemma 2.1}
\begin{proof}
The Laplace-type integral representations for  $P_n(z)$,
\begin{equation}
    P_n(\cosh(\xi)) = \frac{1}{\pi} \int_0^\pi \exp \left[n \log\left(\cosh(\xi) + \sinh(\xi) \cos \theta \right)\right] d\theta. \label{P_laplace_integral}
\end{equation}
The main contribution comes from $\theta =0$ and $\theta =\pi$, so that one can directly apply Laplace's method. For the contribution near $\theta =0$, introduce 
\begin{equation}
\xi - \frac{\sinh \xi}{\exp \xi} \frac{\nu^2}{2} = \log(\cosh \xi + \sinh \xi \cos\theta ),\ \  \theta = \nu - \left(\frac{\sinh \xi}{8\exp \xi} - \frac{1}{24}\right)\nu^3 + \cdots.
\end{equation} 
So for $|\sinh(\xi)| > \epsilon$, the following asymptotic expansion is uniform,
\begin{eqnarray}
    && \int_0^{\pi/2} \exp \left[n\log(\cosh \xi + \sinh \xi \cos\theta ) \right] d\theta \nonumber\\
    &=& \exp \left(n\xi  \right)\int_0^{\infty} \left[1 + \left(\frac{1}{8} - \frac{3\sinh \xi}{8\exp \xi} \right)\nu^2 + \cdots \right] \exp \left( - n\frac{\sinh \xi}{\exp \xi} \frac{\nu^2}{2} \right) d\nu \nonumber\\
    &=& \exp \left(n\xi  \right) \sqrt{\frac{\pi \exp(\xi)}{2n \sinh \xi}} \left[1 - \frac{3}{8n} + \frac{\exp(\xi)}{8n \sinh(\xi)}\right] \left( 1 + O(\frac{1}{n^2}) \right).\label{P_steepest_0}
\end{eqnarray}
Similarly, the contribution near $\theta = \pi$ is given by,
\begin{eqnarray}
    && \int_{\pi/2}^\pi \exp \left[n\log(\cosh \xi + \sinh \xi \cos\theta ) \right] d\theta \nonumber\\
    &=& \exp \left(-n\xi  \right) \sqrt{\frac{\pi \exp(-\xi)}{-2n \sinh \xi}} \left[1 - \frac{3}{8n} - \frac{\exp(-\xi)}{8n \sinh(\xi)} \right] \left( 1 + O(\frac{1}{n^2}) \right).\label{P_steepest_pi}
\end{eqnarray}
Combining \eqref{P_steepest_0} and \eqref{P_steepest_pi} yields,
\begin{eqnarray}
    P_n(\cosh \xi) =&& \sqrt{\frac{i}{2n\pi \sinh \xi}}\left[\left(1 - \frac{1}{4n} + \frac{\coth(\xi)}{8n}\right)\exp\left((n+\frac{1}{2})\xi-i\frac{\pi}{4}\right) \right. \nonumber\\
    && \left.  + \left(1 - \frac{1}{4n} - \frac{\coth(\xi)}{8n}\right)\exp\left(-(n+\frac{1}{2})\xi+i\frac{\pi}{4}\right) + O(\frac{1}{n^2}) \right].
\end{eqnarray}
It is important to note the leading-order term could cancel to an $O(n^{-1})$ term when $\Re(\xi) = O(n^{-2})$, and $\cos\left((n+\frac{1}{2})\Im(\xi)-\frac{\pi}{4}\right)=0$. To guarantee the ratio of the higher-order term to the leading-order term remains uniformly bounded by $O(n^{-1})$, one requires $\bigl|\cosh\!\bigl((n+\tfrac{1}{2})\xi\bigr)\bigr| \;\geq\; \frac{L}{n}$.
\end{proof}

\subsection{Proof of Lemma 2.3}
\begin{proof}
The Laplace-type integral representations for $Q_n(z)$,
\begin{equation}
    Q_n(\cosh(\xi)) =  \int_0^\infty \exp \left[-(n+1) \log\left(\cosh(\xi) + \sinh(\xi) \cosh \theta \right)\right] d\theta. \label{Q_laplace_integral}
\end{equation}
The main contribution comes from the end point $\theta =0$. 
\begin{eqnarray}
    && \int_0^\infty \exp \left[-(n+1) \log\left(\cosh(\xi) + \sinh(\xi) \cosh \theta \right)\right] d\theta \nonumber\\
    &=& \exp \left(-(n+1)\xi  \right)\int_0^{\infty} \exp \left( - (n+1)\frac{\sinh \xi}{\exp \xi} \frac{\theta^2}{2} \right) d\theta \left( 1 + O(\frac{1}{n}) \right) \nonumber\\
    &=& \exp \left(-(n+1)\xi  \right) \sqrt{\frac{\pi \exp(\xi)}{2(n+1) \sinh \xi}}  \left( 1 + O(\frac{1}{n}) \right).\label{Q_steepest_0}
\end{eqnarray}
\end{proof}

\subsection{Proof of Corollary 2.4}
\begin{proof} 
On such Bernstein ellipse, $|\exp(\xi)| = 1 + \frac{M \log n}{n}$, so
    \begin{equation}
        |\cosh(\xi)| \geq \frac{M\log n}{n}, |\sinh(\xi)| \geq \frac{M\log n}{n}.
    \end{equation}
Following the same arguments as in Lemma 2.1 and Lemma 2.3, 
\begin{equation}
    P_n(\cosh \xi) = \sqrt{\frac{2i}{n\pi \sinh \xi}}\cosh\left((n+\frac{1}{2})\xi-i\frac{\pi}{4}\right) \left( 1 + O(\frac{1}{\log n}) \right),
\end{equation}
\begin{equation}
    Q_n(\cosh(\xi)) = \sqrt{\frac{\pi}{2(n+1)i \sinh \xi}}\exp\left(-(n+\frac{1}{2})\xi+i\frac{\pi}{4}\right)  \left( 1 + O(\frac{1}{\log n}) \right),
\end{equation}
are uniform on the Bernstein ellipse $\Omega_B$. So,
\begin{equation}
        \limsup_{n \to \infty} \sup_{z \in \Omega_B} \left| n^{2M}\frac{Q_n(\cosh(\xi))}{P_n(\cosh(\xi))} \right| \leq \lim_{n \to \infty} \frac{n^{2M}}{|1 +\frac{M \log n}{n}|^{2n+1}-1} = 1.
    \end{equation}
\end{proof}

\subsection{Proof of Corollary 2.5}
\begin{proof}
    For $\cosh \xi = \cos \phi +\frac{i}{n}y$, and $ y \in \left[\frac{1}{n}, M \log n\right]$,
    \begin{equation}
        \xi = i\phi + \frac{y}{n\sin \phi} + O(\frac{y^2}{n^2})
    \end{equation}
    Using the asymptotic approximation in Lemma 2.1 and Lemma 2.3,
    \begin{equation}
    \frac{Q_n(b +\frac{i}{n}y)}{P_n(b +\frac{i}{n}y)} = \frac{i\pi +O(\frac{1}{n})}{\left(1 - \frac{1}{4n} - \frac{i\cot\phi}{8n}\right)\exp\left(\frac{2y}{\sin\phi}+i\Psi\right) + \left(1 - \frac{1}{4n} + \frac{i\cot\phi}{8n}\right) }
\end{equation}
Here $\Psi = (2n+1)\phi-\frac{\pi}{2}$. Notice, for $ y \geq 0$, 
\begin{eqnarray}
    &&\left|\left(1 - \frac{1}{4n} - \frac{i\cot\phi}{8n}\right)\exp\left(\frac{2y}{\sin\phi}+i\Psi\right) + \left(1 - \frac{1}{4n} + \frac{i\cot\phi}{8n}\right)\right| \nonumber\\
    &\geq & \sqrt{(1 - \frac{1}{4n})^2 + (\frac{\cot\phi}{8n})^2} \left[\exp\left(\frac{2y}{\sin\phi}\right)-1\right]. 
\end{eqnarray}
So, uniformly in $ y \in \left[\frac{1}{n}, M \log n\right]$, we have
\begin{equation}
    \left|\frac{Q_n(b +\frac{i}{n}y)}{P_n(b +\frac{i}{n}y)} - \frac{i\pi }{\exp\left(\frac{2y}{\sin\phi}+i\Psi\right) + 1 } \right| \leq O\left(\frac{1}{n \left(\exp\left(\frac{2y}{\sin\phi}\right)-1\right)^2}\right)
\end{equation}
\end{proof}

\subsection{Proof of Lemma 3.3}

\begin{proof} {\textit Part (I).} We follow the same steps as in the proof of Lemma 3.2.

\paragraph{Case 1: $k=0$} Prove by control both the integral and the quadrature of $h_n(x)$. 

If $0<\alpha<1$,
\begin{eqnarray}
    |h_n(x)|&=&|x-b|\left|g(\sqrt{(x-b)^2+\frac{1}{n^4}}) - g(|x-b|)\right|
    \nonumber \\ 
    &=&  |x-b| \left|\int_{|x-b|}^{\sqrt{(x-b)^2+\frac{1}{n^4}}} g'(s) ds \right| \nonumber \\ 
    &\leq& B |x-b| \left||x-b|^{\alpha-1} - \left((x-b)^2+\frac{1}{n^4}\right)^{(\alpha-1)/2} \right|\\ 
    &\leq& \frac{B}{n^{2\alpha}} \left[|t|^{\alpha} - |t|\left(t^2+1\right)^{(\alpha-1)/2} \right]
\end{eqnarray}
As $t \to \infty$, $H(t) =|t|^{\alpha} - |t|\left(t^2+1\right)^{(\alpha-1)/2} \to 0$, so the maximum of $H(t)$ is attained at finite $t^*$, and there is only one local maximum on $[0,\infty)$,
\begin{equation}
    \sum_{j=1}^n \omega_j |h_n(x_j)|  \leq \frac{B_2}{n^{2\alpha+1}}, \ \ \int_{-1}^1 |h_n(x)| dx \leq \frac{B_3}{n^{2\alpha+1}}, \ \ R_n[h_n] \leq \frac{B_4}{n^{2\alpha+1}}
\end{equation}

If $\alpha=1$,
\begin{eqnarray}
    &&|h_n(x)|=|x-b|\left|g(\sqrt{(x-b)^2+\frac{1}{n^4}}) - g(|x-b|)\right| \nonumber \\ 
    &\leq& B |x-b| \log\left(1+\frac{1}{n^4(x-b)^2}\right) \\
    &\leq& \frac{B}{n^{2}} |t| \log(1+\frac{1}{t^2}).
\end{eqnarray}
As $t \to \infty$, $H(t) =|t| \log(1+\frac{1}{t^2}) \to 0$, so the maximum of $H(t)$ is attained at finite $t^*$, and there is only one local maximum on $[0,\infty)$,
\begin{equation}
    \sum_{j=1}^n \omega_j |h_n(x_j)|  \leq \frac{B_2' }{n^{3}}, \ \ \int_{-1}^1 |h_n(x)| dx \leq \frac{B_3' }{n^{3}}, \ \ R_n[h_n] \leq \frac{B_4'}{n^{3}}.
\end{equation}

\paragraph{Case 1: $k=1$} Prove by control both the integral and the quadrature of $h_n(x)$.  Use mean value theorem and property of $g$,
\begin{eqnarray}
    |h_n(x)|&=&|x-b|^{2}\left|g(\sqrt{(x-b)^2+\frac{1}{n^4}}) - g(|x-b|)\right| \nonumber \\ 
    &\leq& B |x-b|^{\alpha}\left|\sqrt{(x-b)^2+\frac{1}{n^4}} - |x-b|\right|
\end{eqnarray}
This becomes the same as Eq. (3.25). So, Eq. (3.26)--(3.30) hold.

\paragraph{Case 3: $k\geq 2$} We can again estimate the derivatives.
\begin{eqnarray}
    &&|h_n^{(k)}(x)| \nonumber\\
    &=& \left|\sum_{l=0}^k \frac{k! (k+1)! (x-b)^{l+1}}{l! (l+1)! (k-l)!} \left[g^{(l)}(\sqrt{(x-b)^2+\frac{1}{n^4}})-g^{(l)}(|x-b|)\right]\right| \\
    &\leq& \sum_{l=0}^k \frac{k! (k+1)! }{l! (l+1)! (k-l)!} \left[\sqrt{(x-b)^2+\frac{1}{n^4}}-|x-b|\right] \frac{B}{|x-b|^{1-\alpha}}\label{mvt_g2}\\
    &\leq&  B_5' |x-b|^{\alpha-1}\left[\sqrt{(x-b)^2+\frac{1}{n^4}}-|x-b|\right]. \label{h_n_k_estimate_2}
\end{eqnarray}
This is the same as Eq. (3.32), so Eq. (3.33)--(3.38) are all valid, with slightly different constants.
\end{proof}
\begin{proof}
    {\textit Part (II).} It is sufficient to revisit the contribution from $C_{l1}$.
    \begin{eqnarray}
    &&\int_{C_{l1}} (z-b)^{k+1} g(\sqrt{(z-b)^2+\frac{1}{n^4}}) \frac{Q_n(z)}{P_n(z)} dz\nonumber\\
    &=& \frac{-i^{k+1}\pi}{ n^{k+\alpha+1}}\int_{1/n}^{M\log n} y^{k+1} \frac{  n^{\alpha-1} g(-i\frac{\sqrt{y^2-1/n^2}}{n}) }{\exp\left(\frac{2y}{\sin\phi}+i\Psi\right) + 1} dy + R_1\\
    &=& \frac{-i^{k+1}\pi}{ n^{k+\alpha+1}}\int_{1/n}^{M\log n}  y^{k+1}\frac{ n^{\alpha-1} g(-\frac{iy}{n}) }{\exp\left(\frac{2y}{\sin\phi}+i\Psi\right) + 1} dy + R_1 + R_2\\
    &=& \frac{ -i^{k+1} \pi}{ n^{k+\alpha+1}}\int_{0}^{M\log n}  y^{k+1}\frac{ n^{\alpha-1} g(-\frac{iy}{n}) }{\exp\left(\frac{2y}{\sin\phi}+i\Psi\right) + 1} dy + R_1 + R_2 +R_3.
\end{eqnarray}
Here
\begin{eqnarray}
    R_1 &\leq & O(\frac{1}{ n^{k+3}})\int_{1/n}^{M\log n}  \left|\frac{ y^{k+1} g(-i\frac{\sqrt{y^2-1/n^2}}{n}) }{\sinh^2\left(\frac{2y}{\sin\phi}\right)}  \right| dy \\
    &\leq& O(\frac{1}{ n^{k+\alpha+2}})\int_{1/n}^{M\log n}  \frac{ y^{k+\alpha} }{\sinh^2\left(\frac{2y}{\sin\phi}\right)} dy.
\end{eqnarray}
\begin{equation}
    R_3 \leq O(\frac{1}{ n^{k+\alpha+1}})\int_{0}^{1/n} \frac{ y^{k+\alpha} }{\exp\left(\frac{2y}{\sin\phi}\right) - 1} dy.
\end{equation}
So $R_1$and $R_3$ will be the same as Eq. (3.47)and (3.51).

There is slight difference in $R_2$.

If $0<\alpha<1$, since $z^{2-\alpha} g'(z)$ is bounded on the imaginary axis in a neighborhood of the origin, 
\begin{eqnarray}
    R_2& \leq &O(\frac{1}{ n^{k+2}})\int_{1/n}^{M\log n}  \left|y^{k+1}\frac{  g(-i\frac{\sqrt{y^2-1/n^2}}{n}) - g(-\frac{iy}{n})}{\exp\left(\frac{2y}{\sin\phi}+i\Psi\right) + 1} \right|dy \\
    &\leq & O(\frac{1}{ n^{k+\alpha+1}})\int_{1/n}^{\infty}  y^{k+1}\frac{  y^{\alpha-1} - (y^2-1/n^2)^{(\alpha-1)/2} }{\exp\left(\frac{2y}{\sin\phi}\right) - 1} dy.
\end{eqnarray}

We need to split the integral into two parts,
\begin{eqnarray}
    &&\int_{1/n}^{2/n}  y^{k+1}\frac{  y^{\alpha-1} - (y^2-1/n^2)^{(\alpha-1)/2} }{\exp\left(\frac{2y}{\sin\phi}\right) - 1} dy \nonumber\\
    &\leq& \frac{\sin \phi}{2} \int_{1/n}^{2/n}  y^{k}\left[  y^{\alpha-1} - (y^2-1/n^2)^{(\alpha-1)/2}\right]  dy\\
    &=& \frac{\sin \phi}{2 n^{k+\alpha}} \int_{1}^{2}  t^{k}\left[  t^{\alpha-1} - (t^2-1)^{(\alpha-1)/2}\right]  dt\\
    &=& O(\frac{1}{n^{k+\alpha}}).
\end{eqnarray}
Here we used $\exp(x)-1\geq x$.

\begin{eqnarray}
    &&\int_{2/n}^{\infty}  y^{k+1}\frac{  y^{\alpha-1} - (y^2-1/n^2)^{(\alpha-1)/2} }{\exp\left(\frac{2y}{\sin\phi}\right) - 1} dy \nonumber \\
    &=& \frac{1}{n^2} \int_{2/n}^{\infty}  y^{k+\alpha-2}\frac{ n^2 y^2\left[ 1 - (1-\frac{1}{n^2y^2})^{(\alpha-1)/2} \right]}{\exp\left(\frac{2y}{\sin\phi}\right) - 1}  dy\\
    &\leq & O(\frac{1}{n^{2}}) \int_{2/n}^{\infty} \frac{ y^{k+\alpha-2} }{\exp\left(\frac{2y}{\sin\phi}\right) - 1}  dy\\
    &\leq& \begin{cases}
        O(\frac{1}{n^{k+\alpha}}), \ \ \ \ \  \text{ if } k+\alpha < 2,\\
        O(\frac{1}{n^{2}}),\ \ \ \ \ \ \  \text{ if }  k+\alpha > 2.
    \end{cases}
\end{eqnarray}
Here we used $n^2 y^2\left[ 1 - (1-\frac{1}{n^2y^2})^{(\alpha-1)/2} \right]$ is uniformly bounded on $y \in [2/n,\infty)$.

Add together, for $0<\alpha<1$,
\begin{equation}
    R_2 \leq \begin{cases}
        O(\frac{1}{n^{2k+2\alpha+1}}), \ \ \ \ \  \text{ if } k+\alpha < 2,\\
        O(\frac{1}{n^{k+\alpha+2}}),\ \ \ \ \ \ \  \text{ if }  k+\alpha > 2.
    \end{cases}
\end{equation}

If $\alpha=1$, since $z g'(z)$ is bounded on the imaginary axis in a neighborhood of the origin,
\begin{eqnarray}
    R_2& \leq &O(\frac{1}{ n^{k+2}})\int_{1/n}^{M\log n}  \left|y^{k+1}\frac{  g(-i\frac{\sqrt{y^2-1/n^2}}{n}) - g(-\frac{iy}{n})}{\exp\left(\frac{2y}{\sin\phi}+i\Psi\right) + 1} \right|dy \\
    &\leq & O(\frac{1}{ n^{k+2}})\int_{1/n}^{\infty}  y^{k+1}\frac{  \log y - \log(y^2-1/n^2) }{\exp\left(\frac{2y}{\sin\phi}\right) - 1} dy.
\end{eqnarray}
Similarly, split into two parts,
\begin{eqnarray}
    &&\int_{1/n}^{2/n}  y^{k+1}\frac{  \log y - \log(y^2-1/n^2) }{\exp\left(\frac{2y}{\sin\phi}\right) - 1} dy \nonumber \\
    &\leq & O(\frac{1}{ n^{k+1}})\int_{1}^{2}  t^{k} \left[\log t - \log(t^2-1) \right] dt\\
    &=& O(\frac{1}{ n^{k+1}}).
\end{eqnarray}
\begin{eqnarray}
    &&\int_{2/n}^{\infty}  y^{k+1}\frac{  \log y - \log(y^2-1/n^2) }{\exp\left(\frac{2y}{\sin\phi}\right) - 1} dy \nonumber \\
    &= & \frac{1}{ n^{2}}\int_{2/n}^{\infty}  y^{k-1}\frac{  -n^2 y^2 \log (1-\frac{1}{n^2y^2}) }{\exp\left(\frac{2y}{\sin\phi}\right) - 1} dy \\
    &\leq & O(\frac{1}{ n^{2}})\int_{2/n}^{\infty}  \frac{  y^{k-1}}{\exp\left(\frac{2y}{\sin\phi}\right) - 1} dy\\
    &\leq& \begin{cases}
        O(\frac{\log n}{n^{2}}), \ \ \ \ \  \text{ if } k=1,\\
        O(\frac{1}{n^{2}}),\ \ \ \ \ \ \  \text{ if }  k>1.
    \end{cases}
\end{eqnarray}

Overall, we obtain the same estimate as Eq. (3.50),
\begin{equation}
    R_2 \leq \begin{cases}
        O(\frac{1}{n^{2k+2\alpha+1}}), \ \ \ \ \  \text{ if } k+\alpha < 2,\\
        O(\frac{\log n}{n^{5}}), \ \ \ \ \ \ \ \ \ \  \text{ if } k+\alpha = 2,\\
        O(\frac{1}{n^{k+\alpha+2}}),\ \ \ \ \ \ \  \text{ if }  k+\alpha > 2.
    \end{cases}
\end{equation}
\end{proof}

\bibliographystyle{plain}
\bibliography{SIAM_NA/references}

\end{document}